\documentclass[12pt,reqno]{amsart}

  \usepackage{hyperref}
\usepackage[usenames]{color}
\usepackage[latin1]{inputenc} 
\usepackage{float}

\usepackage[english]{babel} 

\usepackage{amssymb, amsmath}
\usepackage{geometry,graphicx}

\theoremstyle{plain}
\newtheorem{theorem}{Theorem}[section] 

\newtheorem{lemma}[theorem]{Lemma} 

\newtheorem{remark}[theorem]{Remark}
\newtheorem{notation}[theorem]{Notation}
\newtheorem{example}[theorem]{Example}

\numberwithin{equation}{section} 
\newcommand{\Cov}{\mathbb{C}\mathrm{ov}}
\newcommand{\e}{\mathrm{e}}
\newcommand{\dif}{\mathrm{d}}
\newcommand{\leb}{\mathrm{L}}
\newcommand{\support}{\mathrm{support}}

\title[\hfill\protect\parbox{0.95\linewidth}{On the approximation of the probability density function  \\  of the randomized heat equation}]{On the approximation of the probability density function of the  randomized heat equation}

\author{J. Calatayud, J.-C. Cort\'{e}s, M. Jornet}

\begin{document}

\maketitle

\begin{center}
\noindent
\address{Instituto Universitario de Matem\'{a}tica Multidisciplinar,\\
Universitat Polit\`{e}cnica de Val\`{e}ncia,\\
Camino de Vera s/n, 46022, Valencia, Spain\\
email: jucagre@alumni.uv.es; jcortes@imm.upv.es; marcjor@alumni.uv.es
}
\end{center}

\begin{abstract}
In this paper we study the randomized heat equation with homogeneous boundary conditions. The diffusion coefficient is assumed to be a random variable and the initial condition is treated as a stochastic process. The solution of this randomized partial differential equation problem is a stochastic process, which is given by a random series obtained via the classical method of separation of variables. Any stochastic process is determined by its finite-dimensional joint distributions. In this paper, the goal is to obtain approximations to the probability density function of the solution (the first finite-dimensional distributions) under mild conditions. Since the solution is expressed as a random series, we perform approximations of its probability density function. We use two approaches: broadly speaking, first, dealing with the random Fourier coefficients of the random series, and second, taking advantage of the Karhunen-Lo\`{e}ve expansion of the initial condition stochastic process. Finally, several numerical examples illustrating the potentiality of our findings with regard to  both approaches are presented. 
\end{abstract}

\section{Introduction}
Differential equations governing real phenomena often contain some mathematical terms (e.g. initial/boundary condition, source term, coefficients),  referred to as model parameters, that characterize physical features of the  problem and its environment. In practice, these terms must be determined from sampling and/or experimentally. Hence they contain  errors coming from different sources such as the lack of accuracy in sampling and/or measurements and  the inherent uncertainty usually met in complex physical phenomena. In that case,  it is more convenient to treat constants and functions playing the role of model parameters as random variables and stochastic processes, respectively. This approach leads to two different class of differential equations with uncertainty, namely Stochastic Differential Equations (SDEs) and Random Differential Equations (RDEs). Although both terms are often used as synonymous, they are distinctly different and  require completely different techniques for analysis and approximation \cite[pp. 97-98]{llibre_smith}.  In the former case, uncertainty is forced by  the differential of  a stochastic process having an irregular sampling behaviour (e.g., continuous but nowhere differentiable such as the differential of  Brownian motion, i.e., the so-called white noise process).  The  analysis of SDEs requires a special calculus, usually  referred to as It\^{o} Calculus, whose cornerstone is It\^{o} Lemma. This important result plays a key role to conduct both theoretical and numerical analysis for both differential and integral It\^{o}-type equations \cite{Oksendal,Kloeden_Platen,Heydari_JCP,Jerez_JCAM}. Under this approach the uncertainty formulated via  the corresponding differential equation is  restricted to specific patterns  (for instance  of gaussian type when noise is driven by white noise). RDEs consist of a direct randomization of all model parameters subject to uncertainty through random variables and/or stochastic processes having a regular trajectories. This approach allows for a wide range of random patterns (binomial, Poisson, hypergeometric, beta, exponential, etc., but also including gaussian distribution). Analysis of RDEs is  based upon the combination of Probability Theory and Newton-Leibniz Calculus, for which powerful tools are well-established. Both facts are very important advantages in favour of  RDEs \cite{Soong}.

In dealing with SDEs/RDEs defined in a complete probability space, say $(\Omega, \mathcal{F},\mathbb{P})$,  as it also happens in the deterministic scenario, the  primary objective is to compute exact or numerically their solution, say $u(x)$, which is a stochastic process instead of a classical function. A distinctive feature of solving SDEs/RDEs, with respect to their deterministic counterpart, is the need to compute relevant probabilistic information of the solution such as the mean function, $\mathbb{E}[u(x)]$, and the variance function, $\mathbb{V}[u(x)]$. While a more and complex ambitious goal is to determine the finite-dimensional probability distributions, particularly the so-called first probability density function, say $f(u,x)$, associated to the solution, since from it one can compute  any one-dimensional statistical moment
\[
\mathbb{E}[(u(x))^k]= 
\int_{-\infty}^{\infty}
u^k f(u,x) \dif u, \qquad k=1,2,\ldots. 
\]
Furthermore, the computation of $f(u,x)$ permits calculating the probability that the solution stochastic process lies within an interval of interest, say $[u_1,u_2]$,
\[
\mathbb{P}[u_1 \leq u(x)(\omega) \leq u_2]= 
\int_{u_1}^{u_2}
 f(u,x) \dif u, \qquad \omega \in \Omega,
\]  
for each $x$ fixed.

The heat equation is a differential statement of thermal energy balance law. It is a basic model to numerous physical phenomena such as diffusion, heat conduction, transport of solutes, etc., but it has also been successfully applied in other apparently unrelated areas like finance to pricing security derivatives traded in the stock market \cite{Thambynayagam, Wilmott}. Impurities and heterogeneity  in the medium (cross section) and error measurements justify the consideration of randomness in both the diffusion coefficient and the initial condition. This motivates us to study the  randomized  heat equation defined on a finite spatial domain whose diffusion coefficient is assumed to be a random variable,  boundary conditions are homogeneous and initial condition is a stochastic process. Different randomizations to heat equation have been studied in the extant literature using different techniques such as generalized polynomial chaos based stochastic Galerkin technique \cite{Jin_Lu_JCP},  homogenization and Monte Carlo approaches \cite{Xu_AMM}, random mean square calculus \cite{JC_AMM_2014}, random collocation method \cite{Wang_IJHMT_2014}, random interval moment method \cite{Wang_IJHMT_2015}, etc.

Our approach is based upon RDEs  and our main goal is to construct reliable approximations to the probability density function of the solution (the first finite-dimensional distributions) under mild conditions. To achieve this target we will combine the application of Random Variable Transformation (RVT) technique and Karhunen-Lo\`{e}ve expansion (KLE). In the context of RDEs, RVT technique has been successfully applied to compute the probability density function of the solution to significant problems in Physics, Biology, etc., assuming specific  distributions for model parameters \cite{Dorini_CNSNS_2016,Dorini_JCP,Selim_AMC_2011,Selim_EPJP_2015} or dealing with general parametric distributions \cite{JC_CNSNS_2014}.  While some recent contributions where the Karhunen-Lo\`{e}ve expansion is applied to solve relevant problems in Physics can be found in \cite{Selim_2013_JQSRT,Xu_AMM_2016}. Other complementary approaches to compute the probability density function of partial differential equations include closure approximations based on functional integral methods \cite{Karniadakis_JCP_2013} and the application of PGD method \cite{Karniadakis_JCP_2016}, for example.

For the sake of completeness,  we  first  introduce the heat problem that will be randomized later, in the deterministic scenario. Then we briefly discuss some interesting issues and results in the deterministic setting that will allow us to compare  better with our findings when dealing with its random formulation. Let us then consider the deterministic heat equation with homogeneous boundary conditions
\begin{equation} \begin{cases} u_t=\alpha^2 u_{xx},\; 0<x<1,\;t>0, \\ u(0,t)=u(1,t)=0,\;t\geq0, \\ u(x,0)=\phi(x),\;0\leq x\leq 1, \end{cases} 
\label{edp_determinista}
\end{equation}
where the diffusion coefficient is $\alpha^2>0$ and the initial condition is given by $\phi(x)$. The formal solution to (\ref{edp_determinista}) is given, using the method of separation of variables, by
\begin{equation}
u(x,t)=\sum_{n=1}^\infty A_n\,\e^{-n^2\pi^2\alpha^2 t}\sin(n\pi x),
\label{sol}
\end{equation}
where the Fourier coefficient
\[ A_n=2\int_0^1 \phi(y) \sin(n\pi y)\,\dif y \]
is understood as a Lebesgue integral. In fact, under simple hypotheses, it can be easily proved that (\ref{sol}) is indeed a classical solution of (\ref{edp_determinista}).
\begin{theorem}
If $\phi$ is continuous on $[0,1]$, piecewise $C^1$ on $[0,1]$ and $\phi(0)=\phi(1)=0$, then (\ref{sol}) is continuous on $[0,1]\times [0,\infty)$, is of class $C^{2,1}$ on $(0,1)\times(0,\infty)$ and is a classical solution of (\ref{edp_determinista}).
\end{theorem}
\begin{proof}
We present a sketch of the proof. Let us see that $\sum_{n=1}^{\infty} |A_n|<\infty$. We work with Fourier series on $[-1,1]$. Since $\phi(0)=\phi(1)=0$, we can extend $\phi$ in an odd way to $[-1,1]$ so that the resulting function is continuous and piecewise $C^1$ on $[-1,1]$. This allows us to differentiate the Fourier series of $\phi(x)$, $\sum_{n=1}^\infty A_n \sin(n\pi x)$, term by term. Thus, the Fourier series of $\phi'(x)$ is $\sum_{n=1}^\infty n\pi A_n\cos(n\pi x)$. Since $\phi'$ is defined and continuous on $[0,1]$ except at a finite number of points, then it is square integrable, therefore $\sum_{n=1}^\infty n^2 A_n^2<\infty$ by Parseval's identity. By Cauchy-Schwarz inequality, 
\[ \sum_{n=1}^\infty |A_n|\leq \left(\sum_{n=1}^\infty n^2 A_n^2\right)^{\frac12}\left(\sum_{n=1}^\infty \frac{1}{n^2}\right)^{\frac12}<\infty.\]
As $|A_n \e^{-n^2 \pi^2 \alpha^2 t}\sin(n\pi x)|\leq |A_n|$, the series (\ref{sol}) converges absolutely and uniformly on $[0,1]\times [0,\infty)$. To check that $u_t=u_{xx}$, we have to compute the derivatives of (\ref{sol}). For example, to compute $u_t$, we notice that for $t\geq t_0>0$
\[ \left|\frac{\partial}{\partial t}\left(A_n \e^{-n^2 \pi^2 \alpha^2 t}\sin(n\pi x)\right)\right|\leq |A_n| n^2 \pi^2 \alpha^2 \e^{-n^2 \pi^2 \alpha^2 t_0},\]
with $\sum_{n=1}^\infty |A_n| n^2 \pi^2 \alpha^2 \e^{-n^2 \pi^2 \alpha^2 t_0}<\infty$. This implies that 
\[u_t(x,t)=\sum_{n=1}^\infty -A_n n^2 \pi^2 \alpha^2 \e^{-n^2 \pi^2 \alpha^2 t}\sin(n\pi x),\]
for $x\in(0,1)$ and $t>0$. The computation of $u_{xx}$ proceeds similarly.
\end{proof}

\bigskip

Now we consider (\ref{edp_determinista}) in a random setting, meaning that we are going to work on an underlying complete probability space $(\Omega,\mathcal{F},\mathbb{P})$, where $\Omega$ is the set of outcomes, that will be generically denoted by $\omega$, $\mathcal{F}$ is a $\sigma$-algebra of events and $\mathbb{P}$ is a probability measure. We consider the diffusion coefficient $\alpha^2(\omega)$ as a positive random variable and the initial condition 
\[ \phi=\{\phi(x)(\omega):0\leq x\leq 1,\,\omega\in\Omega\} \]
as a stochastic process in our probability space. In this way, the solution given in (\ref{sol}) is a stochastic process expressed as a random series,
\begin{equation}
u(x,t)(\omega)=\sum_{n=1}^\infty A_n(\omega)\,\e^{-n^2\pi^2\alpha^2(\omega) t}\sin(n\pi x),
\label{sol2}
\end{equation}
where the random Fourier coefficient
\[ A_n(\omega)=2\int_0^1 \phi(y)(\omega) \sin(n\pi y)\,\dif y \]
is understood as a Lebesgue integral. 

\begin{notation}
Throughout this paper we will work with Lebesgue spaces. Remember that, if $(S,\mathcal{A},\mu)$ is a measure space, we denote by $\leb^p(S)$ ($1\leq p<\infty$) the set of measurable functions $f:S\rightarrow\mathbb{R}$ such that $\|f\|_{\leb^p(S)}=(\int_S |f|^p\,\dif \mu)^{1/p}<\infty$. We denote by $\leb^\infty(S)$ the set of measurable functions such that $\|f\|_{\leb^\infty(S)}=\inf\{\sup\{|f(x)|:\,x\in S\backslash N\}:\,\mu(N)=0\}<\infty$. We write a.e. as a brief notation for ``almost every'', which means that some property holds except for a set of measure zero.

Here, we will deal with $S=\mathcal{T}\subseteq\mathbb{R}$ and $\dif\mu=\dif x$ the Lebesgue measure, with $S=\Omega$ and $\mu=\mathbb{P}$ the probability measure, and with $S=\mathcal{T}\times\Omega$ and $\dif\mu=\dif x\times \dif\mathbb{P}$. Notice that $f\in \leb^p(\mathcal{T}\times\Omega)$ if and only if $\|f\|_{\leb^p(\mathcal{T}\times \Omega)}=(\mathbb{E}[\int_{\mathcal{T}} |f(x)|^p\,\dif x])^{1/p}<\infty$. In the particular case of $S=\Omega$ and $\mu=\mathbb{P}$, the brief notation a.s. stands for ``almost surely''.

In this paper, an inequality related to Lebesgue spaces will be frequently used. This inequality is well-known as the generalized H\"{o}lder's inequality, which says that, for any measurable functions $f_1,\ldots,f_m$,
\[ \|f_1\cdots f_m\|_{\leb^1(S)}\leq \|f_1\|_{\leb^{r_1}(S)}\cdots \|f_m\|_{\leb^{r_m}(S)}, \]
where
\[ \frac{1}{r_1}+\cdots+\frac{1}{r_m}=1,\quad 1\leq r_1,\ldots,r_m\leq\infty. \]
When $m=2$, this inequality is simply known as H\"{o}lder's inequality. When $m=2$, $r_1=2$ and $r_2=2$, the inequality receives the name of Cauchy-Schwarz inequality.
\end{notation}

\bigskip 

Notice that if $\phi(\cdot)(\omega)\in \leb^1(0,1)$, then 
\[ |A_n(\omega)\,\e^{-n^2\pi^2\alpha^2(\omega) t}\sin(n\pi x)|\leq 2\|\phi(\cdot)(\omega)\|_{\leb^1(0,1)} \e^{-n^2\pi^2\alpha^2(\omega) t}, \] 
so by the comparison test the random series given in (\ref{sol2}) is a.s. convergent and $u(x,t)(\omega)$ is well-defined, for $0<x<1$ and $t>0$. \ \\

The stochastic process (\ref{sol2}) is a rigorous solution to the randomized problem (\ref{edp_determinista}) in the a.s. and $\leb^2$ setting, more specifically:
\begin{theorem}
The following statements hold:\\
\vspace{-0.5cm}
\begin{itemize}
\item[i)] a.s. solution: Suppose that $\phi\in \leb^2([0,1]\times\Omega)$. Then the random series that defines (\ref{sol2}) converges a.s. for all $x\in [0,1]$ and $t>0$. Moreover,
\[ u_t(x,t)(\omega)=\alpha^2(\omega)\,u_{xx}(x,t)(\omega) \]
for $x\in (0,1)$, $t>0$ and a.e. $\omega$, where the derivatives are understood in the classical sense; $u(0,t)(\omega)=u(1,t)(\omega)=0$ for $t\geq0$ and a.e. $\omega$; and $u(x,0)(\omega)=\phi(x)(\omega)$ for a.e. $x\in [0,1]$ and a.e. $\omega$.

\item[ii)] $\leb^2$ solution: Suppose that $\phi\in \leb^2([0,1]\times\Omega)$ and $0<a\leq\alpha^2(\omega)\leq b$ a.e. $\omega\in\Omega$ for certain $a,b\in\mathbb{R}$. Then the random series that defines (\ref{sol2}) converges in $\leb^2(\Omega)$ for all $x\in [0,1]$ and $t>0$. Moreover,
\[ u_t(x,t)(\omega)=\alpha^2(\omega)\,u_{xx}(x,t)(\omega) \]
for $x\in (0,1)$, $t>0$ and a.e. $\omega$, where the derivatives are understood in the mean square sense (Definition 5.33 in \cite{llibre_powell}); $u(0,t)(\omega)=u(1,t)(\omega)=0$ for $t\geq0$ and a.e. $\omega$; and $u(x,0)(\omega)=\phi(x)(\omega)$ for a.e. $x\in [0,1]$ and a.e. $\omega$.
\end{itemize}
\end{theorem}
\begin{proof}
We present a sketch of the proof in the $\leb^2$ setting (in the classical setting is analogous, but acting pointwise on $\omega$). We have, by Cauchy-Schwarz inequality,
\begin{align*}
\|A_n\|_{\leb^2(\Omega)}\leq {} & \left(4\int_0^1\int_0^1 \mathbb{E}[|\phi(y)||\phi(z)|]\,\dif y\,\dif z\right)^{\frac12} \\
\leq {} & \left(4\int_0^1\int_0^1 \mathbb{E}[\phi(y)^2]^{\frac12}\mathbb{E}[\phi(z)^2]^{\frac12}\,\dif y\,\dif z\right)^{\frac12} \\
\leq {} & 2\left(\left(\int_0^1 \mathbb{E}[\phi(y)^2]^{\frac12}\,\dif y\right)\left(\int_0^1 \mathbb{E}[\phi(z)^2]^{\frac12}\,\dif z\right)\right)^{\frac12} \\
= {} & 2\,\int_0^1\mathbb{E}[\phi(y)^2]^{\frac12}\,\dif y\leq 2\, \|\phi\|_{\leb^2([0,1]\times\Omega)}=:C.
\end{align*}
Then
\[ \|A_n \e^{-n^2\pi^2\alpha^2 t}\sin(n\pi x)\|_{\leb^2(\Omega)}\leq \|A_n\|_{\leb^2(\Omega)}\e^{-n^2\pi^2 a t}\leq C\,\e^{-n^2\pi^2 a t}, \]
so by the comparison test $\sum_{n=1}^\infty\|A_n \e^{-n^2\pi^2\alpha^2 t}\sin(n\pi x)\|_{\leb^2(\Omega)}<\infty$ for $x\in [0,1]$ and $t>0$. For the initial condition, since $\phi(x)(\omega)=\sum_{n=1}^\infty A_n(\omega)\sin(n\pi x)$ in $\leb^2(0,1)$ (and also pointwise at a.e. $x\in [0,1]$ by Carleson's Theorem) for a.e. $\omega$, we have $u(x,0)(\omega)=\phi(x)(\omega)$ for a.e. $x\in [0,1]$ and a.e. $\omega$. Finally, to check that $u_t(x,t)(\omega)=\alpha^2(\omega)\,u_{xx}(x,t)(\omega)$, we have to check the mean square uniform convergence of the series of the mean square derivatives (see Theorem 3.1 in \cite{JC}). For example, for $t\geq t_0>0$ and $x\in [0,1]$,
\[
\left\| \frac{\partial }{\partial t}\left(A_n \e^{-n^2\pi^2\alpha^2 t}\sin(n\pi x)\right)\right\|_{\leb^2(\Omega)}\leq   C\,n^2\pi^2\,b\,\e^{-n^2\pi^2a\, t_0}, \]
with $\sum_{n=1}^\infty n^2\pi^2\,b\,\e^{-n^2\pi^2a\, t_0}<\infty$, so we obtain 
\[ u_t(x,t)(\omega)=\sum_{n=1}^\infty -n^2\pi^2 \alpha^2(\omega) A_n(\omega) \e^{-n^2\pi^2\alpha^2(\omega) t}\sin(n\pi x) \]
with uniform convergence in the sense of $\leb^2(\Omega)$, for $t\geq t_0>0$ and $x\in [0,1]$. Since $t_0>0$ is arbitrary, this holds for $t>0$. To compute $u_{xx}$ we proceed similarly. 
\end{proof}

\bigskip

The main goal of this paper is, under suitable hypotheses, to compute approximations of the probability density function of the solution $u(x,t)(\omega)$ given in (\ref{sol2}), for $0<x<1$ and $t>0$. 

In what follows, we will try to solve the problem of finding out the probability density function via two different approaches: \textit{grosso modo}, first, dealing with the joint density of the vector of random Fourier coefficients $(A_1,\ldots,A_N)$, and second, taking advantage of the Karhunen-Lo\`{e}ve expansion of the stochastic process $\phi$ defining the initial condition.

\section{Computing the probability density function under hypotheses on the random vector $(A_1,\ldots,A_N)$}

First, we will present auxiliary results that will be needed afterwards. 

\begin{lemma}[Random Variable Transformation technique] \label{lema_abscont}
Let $X$ be an absolutely continuous random vector with density $f_X$ and with support $D_X$ contained in an open set $D\subseteq \mathbb{R}^n$. Let $g:D\rightarrow\mathbb{R}^n$ be a $C^1(D)$ function, injective on $D$ such that $Jg(x)\neq0$ for all $x\in D$ ($J$ stands for Jacobian). Let $h=g^{-1}:g(D)\rightarrow\mathbb{R}^n$. Let $Y=g(X)$ be a random vector. Then $Y$ is absolutely continuous with density
\begin{equation}
 f_Y(y)=\begin{cases} f_X(h(y))|Jh(y)|, & \;y\in g(D), \\ 0, & \; y\notin g(D). \end{cases}  
\label{dens_lema}
\end{equation}
\end{lemma}

The proof appears in Lemma 4.12 of \cite{llibre_powell}.

\begin{lemma} \label{lema_norm_cond}
Let $Y$ be a multivariate Gaussian random vector with vector mean $\mu$ and covariance matrix $\Sigma$. Partition 
\[ Y=\begin{pmatrix} Y_1 \\ Y_2 \end{pmatrix}, \quad \mu=\begin{pmatrix} \mu_1 \\ \mu_2 \end{pmatrix}, \quad \Sigma=\begin{pmatrix} \Sigma_{11} & \Sigma_{12} \\ \Sigma_{21} & \Sigma_{22} \end{pmatrix}. \]
Then $Y_1|Y_2=a$ is a multivariate Gaussian random vector with vector mean $\bar{\mu}$ and covariance matrix $\bar{\Sigma}$, where 
\[ \bar{\mu}=\mu_1+\Sigma_{12}\Sigma_{22}^{-1}(a-\mu_2),\quad \bar{\Sigma}=\Sigma_{11}-\Sigma_{12}\Sigma_{22}^{-1}\Sigma_{21}. \]
\end{lemma}

The proof appears in Example 4.51 of \cite{llibre_powell}.

\begin{lemma} \label{lema_norm}
Let $\phi=\{\phi(x)(\omega): 0\leq x\leq1,\,\omega\in\Omega\}$ be a Gaussian process in $\leb^2([0,1]\times \Omega)$. Let 
\[A_n(\omega)=2\int_0^1\phi(y)(\omega)\sin(n\pi y)\,\dif y, \]
where the integral is understood in the Lebesgue sense. Then $(A_1,\ldots,A_N)$ is a multivariate Gaussian random vector, for all $N\geq 1$. Moreover,
\[ \mathbb{E}[A_n]=2\int_0^1 \mathbb{E}[\phi(y)]\sin(n\pi y)\,\dif y, \]
\[ \Cov[A_n,A_m]=4\int_0^1\int_0^1 \Cov[\phi(y),\phi(z)]\sin(n\pi y)\sin(m\pi z)\,\dif y\,\dif z. \]
\end{lemma}
\begin{proof}
First, notice that $A_n$ exists and is a random variable, because by Cauchy-Schwarz inequality
\[ \mathbb{E}\left[\int_0^1 |\phi(y)\sin(n\pi y)|\,\dif y\right]\leq \|\phi\|_{\leb^1([0,1]\times \Omega)}\leq\|\phi\|_{\leb^2([0,1]\times \Omega)}< \infty, \]
and Fubini's Theorem applies.

Now, we want to check that $\sum_{j=1}^N \lambda_j A_j$ is normal, for all $\lambda_1,\ldots,\lambda_N \in \mathbb{R}$ (recall that a random vector $(X_1,\ldots,X_m)$ is multivariate Gaussian if and only if every finite linear combination of its random components is Gaussian \cite{Wong}). We write explicitly this sum:
\[\sum_{j=1}^N \lambda_j A_j(\omega)=2\int_0^1 \phi(y)(\omega) \left(\sum_{j=1}^N \lambda_j\sin(j\pi y)\right)\,\dif y=\int_0^1 \phi(y)(\omega)h_N(y)\,\dif y=:Y_N(\omega), \]
where $h_N(y)=2\sum_{j=1}^N \lambda_j\sin(j\pi y)$. We denote by $C$ the bound $2\sum_{j=1}^N |\lambda_j|$ of $|h_N(y)|$.

We have that $Y_N\in \leb^2(\Omega)$, since by Cauchy-Schwarz inequality
\[\mathbb{E}[Y_N^2]=\mathbb{E}\left[\left(\int_0^1\phi(y)h_N(y)\,\dif y\right)^2\right]\leq C^2\|\phi\|_{\leb^2([0,1]\times \Omega)}^2<\infty.\]

Consider the closed vector subspace
\small
\[V=\overline{\left\{\sum_{k=1}^m\mu_k\phi(y_k)h_N(y_k):\, \mu_1,\ldots,\mu_m\in \mathbb{R},\,y_1,\ldots,y_m\in[0,1],\, m\in\mathbb{N}\right\}}^{\;\leb^2(\Omega)}\subseteq \leb^2(\Omega).\]
\normalsize
Since $\phi$ is a Gaussian process and the limit in $\leb^2(\Omega)$ of normal random variables is again a normal random variable\footnote{Let $\{X_n\}_{n=1}^\infty$ be a sequence of random variables with $X_n\sim\text{Normal}(\mu_n,\sigma_n^2)$, such that there exists its limit in $\leb^2(\Omega)$, $X=\lim_{n\rightarrow\infty} X_n$. To see that $X$ is normally distributed, let $\mu$ and $\sigma^2$ be the expectation and variance of $X$, respectively. We have that $\mu_n\rightarrow\mu$ and $\sigma_n^2\rightarrow\sigma^2$ as $n\rightarrow\infty$, by Cauchy-Schwarz inequality. The characteristic function of $X_n$, $\varphi_{X_n}(t)=\e^{\mathrm{i}\mu_n t-\sigma_n^2 t^2/2}$, tends to the function $\varphi(t)=\e^{\mathrm{i}\mu t-\sigma^2 t^2/2}$. Since $\{X_n\}_{n=1}^\infty$ tends in law to $X$, by L\'{e}vy's Theorem, $\varphi(t)$ is the characteristic function of $X$. Therefore $X\sim\text{Normal}(\mu,\sigma^2)$.}, we conclude that any random variable in $V$ has a normal law. Thus, it suffices to prove that $Y_N\in V$. For that purpose, we will use the theory of orthogonality in Hilbert spaces (remember that $\leb^2(\Omega)$ is a Hilbert space with the inner product of two random variables $X_1$ and $X_2$ defined by $\mathbb{E}[X_1X_2]$). Since $V^{\perp\perp}=\overline{V}=V$, it suffices to show that $Y_N\in V^{\perp\perp}$, that is: for all $X\in V^\perp\subseteq \leb^2(\Omega)$, $\mathbb{E}[X\,Y_N]=0$. 

Let $X\in V^\perp\subseteq \leb^2(\Omega)$. Then $\mathbb{E}[X\phi(y)h_N(y)]=0$ for all $y\in[0,1]$, since $\phi(y)h_N(y)\in V$. Thus,
\[ \mathbb{E}[X \,Y_N]=\mathbb{E}\left[X\int_0^1\phi(y)h_N(y)\,\dif y\right]=\int_0^1 \mathbb{E}[X\phi(y) h_N(y)]\,\dif y=0. \]
Notice that the interchange of $\mathbb{E}$ and $\int_0^1$ in this last expression is justified by Fubini's Theorem, since by Cauchy-Schwarz inequality
\small
\[ \mathbb{E}\left[\int_0^1 |\phi(y)||h_N(y)||X|\,\dif y\right]\leq C\,\mathbb{E}\left[|X|\int_0^1 |\phi(y)|\,\dif y\right]\leq C\,\|X\|_{\leb^2(\Omega)}\|\phi\|_{\leb^2([0,1]\times\Omega)}<\infty. \]
\normalsize

\end{proof}

\begin{notation} When the hypotheses of Lemma \ref{lema_norm} hold, we will denote the covariance matrix of $(A_1,\ldots,A_N)$ by $\Sigma_N$ and the mean vector by $\mu_N$. 
\end{notation}

\begin{notation} Given a random vector $X$, its distribution function will be denoted by $F_X$. If it is absolutely continuous, its probability density will be denoted by $f_X$.
\end{notation}

\begin{notation} \label{uN}
We denote
\[ u_N(x,t)(\omega)=\sum_{n=1}^N A_n(\omega)\,\e^{-n^2\pi^2\alpha^2(\omega) t}\sin(n\pi x), \]
which represents a truncation of (\ref{sol2}). 
\end{notation} 

Now we show the main two theorems of this section. The hypotheses are rather technical. For the sake of clarity, we will comment on them later.

\begin{theorem} \label{teor1}
Let $\{\phi(x):\,0\leq x\leq 1\}$ be a Gaussian process in $\leb^2(\Omega\times [0,1])$. Suppose that $\alpha^2$ and $(A_1,\ldots,A_N)$ are independent and absolutely continuous, for $N\geq1$. Assume that $(\Sigma_N^{-1})_{11}\leq C$ for all $N\geq1$ ($\Sigma_N$ is the covariance matrix of $(A_1,\ldots,A_N)$) and $\sum_{n=m}^\infty \|\e^{-(n^2-2)\pi^2\alpha^2 t}\|_{\leb^1(\Omega)}<\infty$ for certain $m\in\mathbb{N}$. Then the density of $u_N(x,t)(\omega)$,
\small
\begin{align}
 f_{u_N(x,t)}(u)= {} & \int_{\mathbb{R}^N} f_{(A_1,\ldots,A_N)} \left(\frac{\e^{\pi^2\alpha^2t}}{\sin(\pi x)}\left\{u-\sum_{n=2}^N a_n \e^{-n^2\pi^2\alpha^2 t}\sin(n \pi x)\right\}, a_2,\ldots,a_N\right) \nonumber \\
\cdot & f_{\alpha^2}(\alpha^2)\frac{\e^{\pi^2\alpha^2 t}}{\sin(\pi x)}\,\dif a_2\cdots \dif a_N\,\dif \alpha^2, \label{frrr}
\end{align}
\normalsize 
converges in $\leb^\infty(\mathbb{R})$ to a density of the random variable $u(x,t)(\omega)$ given in (\ref{sol2}), for $0<x<1$ and $t>0$.
\end{theorem}

\begin{proof}
Let us see that $\{f_{u_N(x,t)}(u)\}_{N=1}^\infty$ is Cauchy in $\leb^{\infty}(\mathbb{R})$, for $0<x<1$ and $t>0$ fixed. 

Fix two indexes $N>M$. Computing marginals, we know that:
\[ f_{u_N(x,t)}(u)=\int_{\mathbb{R}^N}f_{(u_N(x,t),A_2,\ldots,A_N,\alpha^2)}(u,a_2,\ldots,a_N,\alpha^2)\,\dif a_2\cdots \dif a_N\,\dif \alpha^2, \]
\[ f_{u_M(x,t)}(u)=\int_{\mathbb{R}^N}f_{(u_M(x,t),A_2,\ldots,A_N,\alpha^2)}(u,a_2,\ldots,a_N,\alpha^2)\,\dif a_2\cdots \dif a_N\,\dif \alpha^2, \]
which gives rise to our first estimate
\begin{align*}
|f_{u_N(x,t)}(u)-f_{u_M(x,t)}(u)|\leq {} &  \int_{\mathbb{R}^N}|f_{(u_N(x,t),A_2,\ldots,A_N,\alpha^2)}(u,a_2,\ldots,a_N,\alpha^2) \\
- & f_{(u_M(x,t),A_2,\ldots,A_N,\alpha^2)}(u,a_2,\ldots,a_N,\alpha^2)|\,\dif a_2\cdots \dif a_N\,\dif \alpha^2.  
\end{align*}

Now we compute the two probability density functions from the integrand of this last expression, by making use of Lemma \ref{lema_abscont}. Let
\[g(A_1,\ldots,A_N,\alpha^2)=\left(\sum_{n=1}^N A_n\e^{-n^2\pi^2\alpha^2t}\sin(n\pi x),A_2,\ldots,A_N,\alpha^2\right).\]
In the notation of Lemma \ref{lema_abscont}, $D=\mathbb{R}^{N+1}$, $g(D)=\mathbb{R}^{N+1}$,
\[h(A_1,\ldots,A_N,\alpha^2)=\left(\frac{\e^{\pi^2\alpha^2t}}{\sin(\pi x)}\left\{A_1-\sum_{n=2}^N A_n \e^{-n^2\pi^2\alpha^2 t}\sin(n \pi x)\right\}, A_2,\ldots,A_N,\alpha^2\right) \]
and
\[Jh(A_1,\ldots,A_N,\alpha^2)=\frac{\e^{\pi^2\alpha^2 t}}{\sin(\pi x)}>0.\]
Then
\begin{align*}
{} & f_{(u_N(x,t),A_2,\ldots,A_N,\alpha^2)}(u,a_2,\ldots,a_N,\alpha^2) \\
= & f_{(A_1,\ldots,A_N,\alpha^2)} \left(\frac{\e^{\pi^2\alpha^2t}}{\sin(\pi x)}\left\{u-\sum_{n=2}^N a_n \e^{-n^2\pi^2\alpha^2 t}\sin(n \pi x)\right\}, a_2,\ldots,a_N,\alpha^2\right)\frac{\e^{\pi^2\alpha^2 t}}{\sin(\pi x)}. 
\end{align*}
Similarly, by defining
\[g(A_1,\ldots,A_N,\alpha^2)=\left(\sum_{n=1}^M A_n \e^{-n^2\pi^2\alpha^2t}\sin(n\pi x),A_2,\ldots,A_N,\alpha^2\right), \]
we arrive at
\begin{align*}
{} & f_{(u_M(x,t),A_2,\ldots,A_N,\alpha^2)}(u,a_2,\ldots,a_N,\alpha^2) \\
= & f_{(A_1,\ldots,A_N,\alpha^2)} \left(\frac{\e^{\pi^2\alpha^2t}}{\sin(\pi x)}\left\{u-\sum_{n=2}^M a_n \e^{-n^2\pi^2\alpha^2 t}\sin(n \pi x)\right\}, a_2,\ldots,a_N,\alpha^2\right)\frac{\e^{\pi^2\alpha^2 t}}{\sin(\pi x)}. 
\end{align*}
Thus, using the independence between $\alpha^2$ and $(A_1,\ldots,A_N)$, one gets
\small
\begin{align}
{} & |f_{u_N(x,t)}(u)-f_{u_M(x,t)}(u)| \nonumber \\
\leq & \int_{\mathbb{R}^N} \frac{\e^{\pi^2\alpha^2 t}}{\sin(\pi x)}f_{\alpha^2}(\alpha^2) \bigg| f_{(A_1,\ldots,A_N)} \left(\frac{\e^{\pi^2\alpha^2t}}{\sin(\pi x)}\left\{u-\sum_{n=2}^N a_n \e^{-n^2\pi^2\alpha^2 t}\sin(n \pi x)\right\}, a_2,\ldots,a_N\right) \nonumber \\
- & f_{(A_1,\ldots,A_N)} \left(\frac{\e^{\pi^2\alpha^2t}}{\sin(\pi x)}\left\{u-\sum_{n=2}^M a_n \e^{-n^2\pi^2\alpha^2 t}\sin(n \pi x)\right\}, a_2,\ldots,a_N\right) \bigg| \,\dif a_2\cdots \dif a_N\,\dif \alpha^2. \label{fins_aci} 
\end{align}
\normalsize
We write $f_{(A_1,\ldots,A_N)}(a_1,a_2,\ldots,a_N)=f_{A_1|(A_2,\ldots,A_N)}(a_1|a_2,\ldots,a_N)f_{(A_2,\ldots,A_N)}(a_2,\ldots,a_N)$:
\small
\begin{align}
{} & |f_{u_N(x,t)}(u)-f_{u_M(x,t)}(u)| \nonumber \\
\leq & \int_{\mathbb{R}^N} \frac{\e^{\pi^2\alpha^2 t}}{\sin(\pi x)}f_{\alpha^2}(\alpha^2) \bigg| f_{A_1|(A_2,\ldots,A_N)} \left(\frac{\e^{\pi^2\alpha^2t}}{\sin(\pi x)}\left\{u-\sum_{n=2}^N a_n \e^{-n^2\pi^2\alpha^2 t}\sin(n \pi x)\right\}\big| a_2,\ldots,a_N\right) \nonumber \\
- & f_{A_1|(A_2,\ldots,A_N)} \left(\frac{\e^{\pi^2\alpha^2t}}{\sin(\pi x)}\left\{u-\sum_{n=2}^M a_n \e^{-n^2\pi^2\alpha^2 t}\sin(n \pi x)\right\}\big| a_2,\ldots,a_N\right) \bigg| \nonumber \\
\cdot & f_{(A_2,\ldots,A_N)}(a_2,\ldots,a_N) \,\dif a_2\cdots \dif a_N\,\dif \alpha^2. \label{a1cond}
\end{align}
\normalsize
Partition
\[\mu_N=\begin{pmatrix} (\mu_N)_1\\ \mu_N^{(2)}\end{pmatrix},\quad \Sigma_N=\begin{pmatrix} (\Sigma_N)_{11} & (\sigma_N^{(2)})^{T}\\ \sigma_N^{(2)}& \Sigma_N^{(2)}\end{pmatrix},      \]
where $(\mu_N)_1\in\mathbb{R}$, $\mu_N^{(2)}\in\mathbb{R}^{N-1}$, $(\Sigma_N)_{11}\in\mathbb{R}$, $\Sigma_N^{(2)}\in\mathbb{R}^{(N-1)\times (N-1)}$ and $\sigma_N^{(2)}\in\mathbb{R}^{N-1}$, and denote $\alpha_N=(a_2,\ldots,a_N)^{T}$. By Lemma \ref{lema_norm_cond}, $A_1|(A_2=a_2,\ldots,A_N=a_N)$ follows a normal distribution with mean $\bar{\mu}_N$ and variance $\bar{\Sigma}_N$, where 
\[\bar{\mu}_N=(\mu_N)_1+(\sigma_N^{(2)})^T(\Sigma_N^{(2)})^{-1}(\alpha_N-\mu_N^{(2)}),\quad \bar{\Sigma}_N=(\Sigma_N)_{11}-(\sigma_N^{(2)})^T(\Sigma_N^{(2)})^{-1}\sigma_N^{(2)}.\]
Write explicitly
\[ f_{A_1|(A_2,\ldots,A_N)}(x|a_2,\ldots,a_N)=\frac{1}{\sqrt{2\pi\bar{\Sigma}_N}}\e^{-\frac{(x-\bar{\mu}_N)^2}{2\bar{\Sigma}_N}}. \]
The maximum on $\mathbb{R}$ of
\[ \left|\frac{\dif}{\dif x}f_{A_1|(A_2,\ldots,A_N)}(x|a_2,\ldots,a_N)\right|=\frac{1}{\sqrt{2\pi\bar{\Sigma}_N}}\frac{|x-\bar{\mu}_N|}{\bar{\Sigma}_N} \e^{-\frac{(x-\bar{\mu}_N)^2}{2\bar{\Sigma}_N}} \]
is $\e^{-1/2}(1/\sqrt{2\pi})(1/\bar{\Sigma}_N)$. To bound $1/\bar{\Sigma}_N$, notice that
\[ 1/\bar{\Sigma}_N=\frac{1}{(\Sigma_N)_{11}-(\sigma_N^{(2)})^T(\Sigma_N^{(2)})^{-1}\sigma_N^{(2)}}=(\Sigma_N^{-1})_{11}\leq C, \footnote{It is well-known that the inverse of a matrix divided in four blocks is the following:
\[ \begin{pmatrix} A & B \\ C & D\end{pmatrix}^{-1}=\begin{pmatrix} (A-BD^{-1}C)^{-1} & -(A-BD^{-1}C)^{-1} BD^{-1} \\ -D^{-1}C(A-BD^{-1}C)^{-1} & D^{-1}+D^{-1}C(A-BD^{-1}C)^{-1}BD^{-1} \end{pmatrix}, \]
whenever the inverses from the right-hand side of the equality exist (see Theorem 2.1 (ii) in \cite{xinesos}).} \]
where $C$ is independent of $x$, $\alpha_N=(a_2,\ldots,a_N)^T$ and $N$ by hypothesis. By the Mean Value Theorem,
\[ |f_{A_1|(A_2,\ldots,A_N)}(x_1|a_2,\ldots,a_N)-f_{A_1|(A_2,\ldots,A_N)}(x_2|a_2,\ldots,a_N)|\leq L |x_1-x_2|, \]
where $L=\e^{-1/2}C/\sqrt{2\pi}$ represents the Lipschitz constant.

Going back to (\ref{a1cond}) and using the definition of the expectation via an integral and the independence of $\alpha^2$ and $(A_1,\ldots,A_N)$, one deduces 
\begin{align}
{} & |f_{u_N(x,t)}(u)-f_{u_M(x,t)}(u)| \label{igual} \\
\leq & L\int_{\mathbb{R}^N} \frac{\e^{2\pi^2\alpha^2 t}}{\sin^2(\pi x)}\left(\sum_{n=M+1}^N |a_n|\e^{-n^2\pi^2\alpha^2 t}\right)f_{\alpha^2}(\alpha^2)f_{(A_2,\ldots,A_N)}(a_2,\ldots,a_N)\,\dif a_2\cdots \dif a_N\,\dif \alpha^2 \nonumber \\
= & \frac{L}{\sin^2(\pi x)} \sum_{n=M+1}^N \mathbb{E}[|A_n|\e^{-(n^2-2)\pi^2\alpha^2 t}]= \frac{L}{\sin^2(\pi x)} \sum_{n=M+1}^N \mathbb{E}[|A_n|]\,\mathbb{E}[\e^{-(n^2-2)\pi^2\alpha^2 t}] \nonumber \\
\leq & \frac{2\,\|\phi\|_{\leb^2([0,1]\times\Omega)}\, L}{\sin^2(\pi x)}\sum_{n=M+1}^N \| \e^{-(n^2-2)\pi^2\alpha^2 t}\|_{\leb^1(\Omega)}. \nonumber
\end{align}
In the last inequality, we used the following bound:
\begin{align*} \mathbb{E} {} & [|A_n|]\leq \mathbb{E}[A_n^2]^{\frac12}\leq \left(4\int_0^1\int_0^1 \mathbb{E}[|\phi(y)||\phi(z)|]\,\dif y\,\dif z\right)^{\frac12} \\
\leq & \left(4\int_0^1\int_0^1 \mathbb{E}[\phi(y)^2]^{\frac12}\mathbb{E}[\phi(z)^2]^{\frac12}\,\dif y\,\dif z\right)^{\frac12}=2\,\int_0^1\mathbb{E}[\phi(y)^2]^{\frac12}\,\dif y\leq 2\, \|\phi\|_{\leb^2([0,1]\times\Omega)}. 
\end{align*}

Since we assume that $\sum_{n=m}^\infty \|\e^{-(n^2-2)\pi^2\alpha^2 t}\|_{\leb^1(\Omega)}<\infty$, we conclude that $\{f_{u_N(x,t)}(u)\}_{N=1}^\infty$ is Cauchy in $\leb^{\infty}(\mathbb{R})$. 

Let 
\[ g_{x,t}(u)=\lim_{N\rightarrow\infty}f_{u_N(x,t)}(u),\quad u\in\mathbb{R}. \]
We need to prove that $g_{x,t}$ is a density of the random variable $u(x,t)$ given in (\ref{sol2}).

First, notice that $g_{x,t}\in \leb^1(\mathbb{R})$, since by Fatou's Lemma
\[ \int_{\mathbb{R}} g_{x,t}(u)\,\dif u=\int_{\mathbb{R}} \lim_{N\rightarrow\infty}f_{u_N(x,t)}(u)\,\dif u\leq \liminf_{N\rightarrow\infty} \underbrace{\int_{\mathbb{R}} f_{u_N(x,t)}(u)\,\dif u}_{=1}=1<\infty. \]

On the other hand, for $0<x<1$, $t>0$ and a.e. $\omega\in\Omega$, the series in (\ref{sol2}) converges in $\mathbb{R}$, therefore $u_N(x,t)(\omega)\rightarrow u(x,t)(\omega)$ as $n\rightarrow\infty$ a.s., which implies convergence in law: 
\[F_{u_N(x,t)}(u)\stackrel{n\rightarrow\infty}{\longrightarrow} F_{u(x,t)}(u), \]
for all $u\in \mathbb{R}$ which is a point of continuity of $F_{u(x,t)}$.

Since $f_{u_N(x,t)}$ is the density of $u_N(x,t)$,
\[ F_{u_N(x,t)}(u)=F_{u_N(x,t)}(u_0)+\int_{u_0}^u f_{u_N(x,t)}(v)\,\dif v. \]
If $u$ and $u_0$ are points of continuity of $F_{u(x,t)}$, taking limits when $N\rightarrow\infty$ we get
\[ F_{u(x,t)}(u)=F_{u(x,t)}(u_0)+\int_{u_0}^u g_{x,t}(v)\,\dif v \]
(recall that $\{f_{u_N(x,t)}\}_{N=1}^\infty$ converges to $g_{x,t}$ in $\leb^{\infty}(\mathbb{R})$, so we can interchange the limit and the integral). As the points of discontinuity of $F_{u(x,t)}$ are countable and $F_{u(x,t)}$ is right continuous, we obtain 
\[ F_{u(x,t)}(u)=F_{u(x,t)}(u_0)+\int_{u_0}^u g_{x,t}(v)\,\dif v \]
for all $u_0$ and $u$ in $\mathbb{R}$.

Thus, $g_{x,t}=f_{u(x,t)}$ is a density for $u(x,t)$, as wanted. 

\end{proof}

\begin{theorem} \label{teor2}
Let $\{\phi(x):\,0\leq x\leq 1\}$ be a process in $\leb^2(\Omega\times [0,1])$. Suppose that $\alpha^2$, $A_1$ and $(A_2,\ldots,A_N)$ are independent and absolutely continuous, for $N\geq2$. Suppose that the probability density function $f_{A_1}$ is Lipschitz on $\mathbb{R}$. Assume that $\sum_{n=m}^\infty \|\e^{-(n^2-2)\pi^2\alpha^2 t}\|_{\leb^1(\Omega)}<\infty$ for certain $m\in\mathbb{N}$. Then the density of $u_N(x,t)(\omega)$,
\small
\begin{align}
 f_{u_N(x,t)}(u)= {} & \int_{\mathbb{R}^N} f_{A_1} \left(\frac{\e^{\pi^2\alpha^2t}}{\sin(\pi x)}\left\{u-\sum_{n=2}^N a_n \e^{-n^2\pi^2\alpha^2 t}\sin(n \pi x)\right\}\right)f_{(A_2,\ldots,A_N)} (a_2,\ldots,a_N) \nonumber \\
\cdot & f_{\alpha^2}(\alpha^2)\frac{\e^{\pi^2\alpha^2 t}}{\sin(\pi x)}\,\dif a_2\cdots \dif a_N\,\dif \alpha^2, \label{fr}
\end{align}
\normalsize 
converges in $\leb^\infty(\mathbb{R})$ to a density of the random variable $u(x,t)(\omega)$ given in (\ref{sol2}), for $0<x<1$ and $t>0$.
\end{theorem}

\begin{proof}
The proof goes the same as Theorem \ref{teor1} up to expression (\ref{fins_aci}). In (\ref{fins_aci}), we use the independence of $\alpha^2$, $A_1$ and $(A_2,\ldots,A_N)$:
\small
\begin{align*}
{} & |f_{u_N(x,t)}(u)-f_{u_M(x,t)}(u)| \\
\leq & \int_{\mathbb{R}^N} \frac{\e^{\pi^2\alpha^2 t}}{\sin(\pi x)}f_{\alpha^2}(\alpha^2) \bigg| f_{A_1} \left(\frac{\e^{\pi^2\alpha^2t}}{\sin(\pi x)}\left\{u-\sum_{n=2}^N a_n \e^{-n^2\pi^2\alpha^2 t}\sin(n \pi x)\right\}\right) \\
- & f_{A_1} \left(\frac{\e^{\pi^2\alpha^2t}}{\sin(\pi x)}\left\{u-\sum_{n=2}^M a_n \e^{-n^2\pi^2\alpha^2 t}\sin(n \pi x)\right\}\right) \bigg| \\
\cdot & f_{(A_2,\ldots,A_N)}(a_2,\ldots,a_N) \,\dif a_2\cdots \dif a_N\,\dif \alpha^2. 
\end{align*}
\normalsize
Denoting by $L$ the Lipschitz constant of $f_{A_1}$, we proceed exactly as in the previous proof from step (\ref{igual}).
\end{proof}

\begin{remark} The density of $u_N(x,t)(\omega)$, as we saw in the proof of Theorem \ref{teor1}, is (\ref{frrr}). In the case that $\phi$ is a Gaussian process, we know by Lemma \ref{lema_norm} that $(A_1,\ldots,A_N)$ is a multivariate Gaussian random vector. In the case that $(A_1,\ldots,A_N)$ is absolutely continuous (that is, $\det(\Sigma_N)>0$), $\alpha^2$ is absolutely continuous and independent, we can compute $f_{u_N(x,t)}(u)$. Under the assumptions of Theorem \ref{teor1} or Theorem \ref{teor2}, $f_{u_N(x,t)}(u)$ is approximately a density of (\ref{sol2}) for large $N$.
\end{remark}

\begin{remark} In Theorem \ref{teor2}, if $\phi$ is a Gaussian process then the hypothesis $f_{A_1}$ Lipschitz on $\mathbb{R}$ holds, since the density function of a normal distribution is Lipschitz (its derivative is bounded on $\mathbb{R}$). On the other hand, the common hypothesis in Theorem \ref{teor1} and Theorem \ref{teor2}, $\sum_{n=m}^\infty \|\e^{-(n^2-2)\pi^2\alpha^2 t}\|_{\leb^1(\Omega)}<\infty$ for certain $m\in\mathbb{N}$, holds for instance when $\alpha^2(\omega)\geq a>0$ for a.e. $\omega\in\Omega$.
\end{remark}

\begin{remark} The hypothesis $(\Sigma_N^{-1})_{11}\leq C$ for all $N\geq1$ would be very difficult to check in practice. Using the usual formula for the inverse of a matrix using the procedure of the ``adjoint matrix'', $(\Sigma_N^{-1})_{11}=\det(\Sigma_N^{(2)})/\det(\Sigma_N)$, where $\Sigma_N^{(2)}$ is the submatrix of $\Sigma_N$ obtained after deleting the first row and column from $\Sigma_N$. There are upper bounds for the determinant of a symmetric positive-definite matrix, for example, Hadamard's Determinant Theorem says that the determinant of a symmetric positive-definite matrix is bounded above by the product of its diagonal elements. However, no simple lower bounds are known for the determinant, so, at least to our knowledge, it is not possible to ensure that $(\Sigma_N^{-1})_{11}\leq C$ for all $N\geq1$ in general. 
\end{remark}

\begin{remark} \label{brownia_no_funciona} Lemma \ref{lema_norm} says that, if $\phi$ is a Gaussian process, then $(A_1,\ldots,A_N)$ is a multivariate Gaussian random vector for $N\geq1$. However, this does not mean that $(A_1,\ldots,A_N)$ is absolutely continuous, since it could be possible that $\det(\Sigma_N)=0$. 

This happens for example when $\phi(x)=W(x)$, where $W$ is a standard Brownian motion on $[0,1]$. Indeed, taking into account that $\Cov[W(y),W(z)]=\min\{y,z\}$,
\[ \Cov[A_n,A_m]=4\int_0^1\int_0^1 \min\{y,z\}\sin(n\pi y)\sin(m\pi z)\,\dif y\,\dif z=\frac{4(-1)^{m+n}}{m\,n\pi^2} \]
for $1\leq n,m\leq N$, so $\Sigma_N=4/\pi^2 (-1,1/2,-1/3,\ldots)^T(-1,1/2,-1/3,\ldots)$, and since $\mathrm{rank}(\mathrm{A}\mathrm{B})\leq\min\{\mathrm{rank}(\mathrm{A}),\mathrm{rank}(\mathrm{B})\}$ for any general matrices $\mathrm{A}$ and $\mathrm{B}$ that can be multiplied, we have $\mathrm{rank}(\Sigma_N)=1$. Hence, for $N\geq 2$, $\det(\Sigma_N)=0$, and Theorem \ref{teor1} and Theorem \ref{teor2} cannot be applied when $\phi$ is a Brownian motion. 

Nevertheless, the fact that the theorems cannot be applied when dealing with the Brownian motion is obvious, since problem (\ref{edp_determinista}) tells us that $\phi(1)=u(1,0)=0$, which is not true for a Brownian motion. Thus, it makes no sense to model the initial condition as a Brownian process.
\end{remark}

\begin{remark} \label{bridge_funciona} If $\phi(x)=B(x)$, where $B$ is a standard Brownian bridge on $[0,1]$, then $A_1,A_2,\ldots$ are independent, and we are in position of applying Theorem \ref{teor2}. Indeed, taking into account that $\Cov[B(y),B(z)]=\min\{y,z\}-yz$,
\[ \Cov[A_n,A_m]=4\int_0^1\int_0^1 (\min\{y,z\}-y\,z)\sin(n\pi y)\sin(m\pi z)\,\dif y\,\dif z=\begin{cases} 0,&\;n\neq m, \\ \frac{2}{n^2\pi^2},&\;n=m, \end{cases} \]
for $1\leq n,m\leq N$, and since $(A_1,\ldots,A_N)$ is multivariate Gaussian for all $N\geq1$ by Lemma \ref{lema_norm}, the independence of $A_1,A_2,\ldots$ follows.

Recall that the Brownian bridge has a zero value at $x=1$, so it does make sense to model the initial condition via a Brownian bridge, as opposed to Brownian motion.

Continuing with the computations, we have $A_1,A_2,\ldots$ independent and $A_n\sim\text{Normal}(0,2/(n^2\pi^2))$ for $n\geq1$, so
\[ f_{(A_1,\ldots,A_N)}(a_1,\ldots,a_n)=\left(\frac{\sqrt{\pi}}{2}\right)^N\,\prod_{n=1}^N n\,\e^{-\frac{n^2\pi^2 a_n^2}{4}}. \]
Thus,
\begin{align}
 f_{u_N(x,t)}(u)= {} & \left(\frac{\sqrt{\pi}}{2}\right)^N \int_{\mathbb{R}^N} \e^{-\frac{\pi^2}{4}\frac{\e^{2\pi^2\alpha^2 t}}{\sin^2(\pi x)}\left\{u-\sum_{n=2}^N a_n \e^{-n^2\pi^2\alpha^2 t}\sin(n \pi x)\right\}^2}\left(\prod_{n=2}^N n\,\e^{-\frac{n^2\pi^2 a_n^2}{4}}\right) \nonumber \\
\cdot & f_{\alpha^2}(\alpha^2)\frac{\e^{\pi^2\alpha^2 t}}{\sin(\pi x)}\,\dif a_2\cdots \dif a_N\,\dif \alpha^2. \label{fbb1}
\end{align}
\end{remark}

\section{Computing the probability density function under hypotheses on the Karhunen-Lo\`{e}ve expansion of \texorpdfstring{$\phi$}{phi}}

In this section we will use two lemmas. The first one is Lemma \ref{lema_abscont}. The second lemma is Karhunen-Lo\`{e}ve Theorem, which is proved in Theorem 5.28 of \cite{llibre_powell}.

\begin{lemma}[Karhunen-Lo\`{e}ve Theorem] \label{KLlemma}
Consider a process $\{X(t):\,t\in\mathcal{T}\}$ in $\leb^2(\mathcal{T}\times\Omega)$. Then
\[ X(t,\omega)=\mu(t)+\sum_{j=1}^\infty \sqrt{\nu_j}\,\phi_j(t)\xi_j(\omega), \]
where the sum converges in $\leb^2(\mathcal{T}\times\Omega)$, $\mu(t)=\mathbb{E}[X(t)]$, $\{\phi_j\}_{j=1}^\infty$ is an orthonormal basis of $\leb^2(\mathcal{T})$, $\{(\nu_j,\phi_j)\}_{j=1}^\infty$ is the set of pairs of (nonnegative) eigenvalues and eigenfunctions of the operator
\begin{equation}
 \mathcal{C}:\leb^2(\mathcal{T})\rightarrow \leb^2(\mathcal{T}),\; \mathcal{C}f(t)=\int_{\mathcal{T}} \Cov[X(t),X(s)]f(s)\,\dif s, 
 \label{karhC}
\end{equation}
and $\{\xi_j\}_{j=1}^\infty$ is a sequence of random variables with zero expectation, unit variance and pairwise uncorrelated. Moreover, if $\{X(t):\,t\in\mathcal{T}\}$ is a Gaussian process, then $\{\xi_j\}_{j=1}^\infty$ are independent and Gaussian.
\end{lemma}

\begin{remark}
When the operator $\mathcal{C}$ defined in (\ref{karhC}) has only a finite number of nonzero eigenvalues, then the process $X$ of Lemma \ref{KLlemma} can be expressed as a finite sum:
\[ X(t,\omega)=\mu(t)+\sum_{j=1}^I \sqrt{\nu_j}\,\phi_j(t)\xi_j(\omega). \]
In the subsequent development, we will write the data stochastic process $\phi$ via its Karhunen-Lo\`{e}ve expansion. The summation symbol in the expansion will be always written up to $\infty$ (the most difficult case), although it could be possible that its corresponding covariance integral operator $\mathcal{C}$ has only a finite number of nonzero eigenvalues. In such a case, in expression (\ref{NM3}) one has to interpret that the vector $(\xi_1,\ldots,\xi_{M-1})$ finishes at $M-1=I<\infty$, whereas the other index $N$ grows up to infinity. From (\ref{NM3}), the modifications are straightforward and easier than for $I=\infty$. Details are left to the reader.
\end{remark}

Take the truncation of (\ref{sol2}) from Notation \ref{uN}:
\[ u_N(x,t)(\omega)=\sum_{n=1}^N A_n(\omega)\e^{-n^2\pi^2\alpha^2(\omega)t}\sin(n\pi x), \]
where
\[ A_n(\omega)=2\int_0^1 \phi(y)(\omega)\sin(n\pi y)\,\dif y. \]

If $\phi\in \leb^2([0,1]\times\Omega)$, we can compute its Karhunen-Lo\`{e}ve expansion
\begin{equation}
 \phi(x)(\omega)=\mu_\phi(x)+\sum_{m=1}^\infty \sqrt{\nu_m} \,\phi_m(x)\xi_m(\omega), 
\label{klphi}
\end{equation}
where $\mu_\phi(x)=\mathbb{E}[\phi(x)]$ and $\{(\nu_m,\phi_m)\}_{m=1}^\infty$ is the set of pairs of (nonnegative) eigenvalues and eigenfunctions of the operator 
\[ \mathcal{C}:\leb^2(0,1)\rightarrow \leb^2(0,1),\; \mathcal{C}f(t)=\int_0^1\Cov[\phi(t),\phi(s)]f(s)\,\dif s. \]
We will assume that the sequence of pairs $\{(\nu_m,\phi_m)\}_{m=1}^\infty$ does not have a particular ordering. In practice, the ordering will be chosen so that the hypotheses of Theorem \ref{teor3} stated later on are satisfied (for example, if we say in the theorem that $\xi_1$ and $\phi_1$ have to satisfy a certain condition, then we can reorder the pairs of eigenvalues and eigenfunctions and the random variables $\xi_1,\xi_2,\ldots$ so that $\xi_1$ and $\phi_1$ satisfy the condition).

If we truncate the Karhunen-Lo\`{e}ve expression of $\phi$ up to an index $M-1$, we obtain a new truncation of (\ref{sol2}):
\small
\[ u_{N,M}(x,t)(\omega)=\sum_{n=1}^N \left\{2\int_0^1 \left(\mu_\phi(y)+\sum_{m=1}^{M-1}\sqrt{\nu_m}\,\phi_m(y)\xi_m(\omega)\right)\sin(n\pi y)\,\dif y\right\}\e^{-n^2\pi^2\alpha^2(\omega)t}\sin(n\pi x). \]
\normalsize

Using Lemma \ref{lema_abscont}, we compute the density $f_{u_{N,M}(x,t)}(u)$ of the random variable $u_{N,M}(x,t)(\omega)$. In order to simplify the notation, we introduce a new operator,
\small
\[ T_N(f)(x,t,\alpha^2)=\sum_{n=1}^N\hat{f}(n)\e^{-n^2\pi^2\alpha^2t}\sin(n\pi x), \]
\normalsize
where $\hat{f}(n)=\int_0^1 f(y)\sin(n\pi y)\,\dif y$,
for $f\in \leb^2(0,1)$. With this new notation, $u_{N,M}(x,t)(\omega)$ becomes
\small
\[ u_{N,M}(x,t)(\omega)=2\,T_N(\mu_\phi)(x,t,\alpha^2(\omega))+2\,\sum_{m=1}^{M-1}T_N(\phi_m)(x,t,\alpha^2(\omega))\sqrt{\nu_m}\,\xi_m(\omega). \]
\normalsize
In the notation of Lemma \ref{lema_abscont}, 
\small
\[ g(\xi_1,\ldots,\xi_{M-1},\alpha^2)=\left(2\,T_N(\mu_\phi)(x,t,\alpha^2)+2\,\sum_{m=1}^{M-1}T_N(\phi_m)(x,t,\alpha^2)\sqrt{\nu_m}\,\xi_m,\xi_2,\ldots,\xi_{M-1},\alpha^2\right), \]
\normalsize
$D=\mathbb{R}^{M-1}\times \mathcal{D}_N$, where $\mathcal{D}_N=\{\alpha^2>0:\,T_N(\phi_1)(x,t,\alpha^2)\neq0\}$, $g(D)=\mathcal{R}^{M-1}\times \mathcal{D}_N$,
\small
\[ h(\xi_1,\ldots,\xi_{M-1},\alpha^2)=\left(\frac{\frac{\xi_1}{2}-T_N(\mu_{\phi})(x,t,\alpha^2)-\sum_{m=2}^{M-1}T_N(\phi_m)(x,t,\alpha^2)\sqrt{\nu_m}\,\xi_m}{\sqrt{\nu_1}\,T_N(\phi_1)(x,t,\alpha^2)},\xi_2,\ldots,\xi_{M-1},\alpha^2 \right) \]
\normalsize
and 
\[ Jh(\xi_1,\ldots,\xi_{M-1},\alpha^2)=\frac{1}{2\sqrt{\nu_1}\,T_N(\phi_1)(x,t,\alpha^2)}\neq 0. \]

Computing marginals,
\footnotesize
\begin{align}
{} & f_{u_{N,M}(x,t)}(u) \nonumber \\
= & \int_{\mathbb{R}^{M-2}\times\mathcal{D}_N}f_{(\xi_1,\ldots,\xi_{M-1},\alpha^2)}\bigg(\frac{\frac{u}{2}-T_N(\mu_{\phi})(x,t,\alpha^2)-\sum_{m=2}^{M-1}T_N(\phi_m)(x,t,\alpha^2)\sqrt{\nu_m}\,\xi_m}{\sqrt{\nu_1}\,T_N(\phi_1)(x,t,\alpha^2)},\xi_2,\ldots,\xi_{M-1},\alpha^2 \bigg) \nonumber \\
\cdot & \frac{1}{2\sqrt{\nu_1}\,|T_N(\phi_1)(x,t,\alpha^2)|}\,\dif \xi_2\cdots \dif \xi_{M-1}\,\dif \alpha^2. \label{NM3} 
\end{align}
\normalsize

To simplify this function and without loss of generality, we put $N=M$ so that we have a unique index:
\footnotesize
\begin{align}
{} & f_{u_{N,N}(x,t)}(u) \nonumber \\
= & \int_{\mathbb{R}^{N-2}\times\mathcal{D}_N}f_{(\xi_1,\ldots,\xi_{N-1},\alpha^2)}\bigg(\frac{\frac{u}{2}-T_N(\mu_{\phi})(x,t,\alpha^2)-\sum_{m=2}^{N-1}T_N(\phi_m)(x,t,\alpha^2)\sqrt{\nu_m}\,\xi_m}{\sqrt{\nu_1}\,T_N(\phi_1)(x,t,\alpha^2)},\xi_2,\ldots,\xi_{N-1},\alpha^2 \bigg) \nonumber \\
\cdot & \frac{1}{2\sqrt{\nu_1}\,|T_N(\phi_1)(x,t,\alpha^2)|}\,\dif \xi_2\cdots \dif \xi_{N-1}\,\dif \alpha^2. \label{fnn}
\end{align}
\normalsize

In the following theorem, we establish conditions under which $\{f_{u_{N,N}(x,t)}\}_{N=1}^\infty$ converges to a density of the random variable $u(x,t)(\omega)$ defined in (\ref{sol2}).

\begin{theorem} \label{teor3}
Let $\{\phi(x):0\leq x\leq 1\}$ be a process in $\leb^2([0,1]\times\Omega)$ such that its Karhunen-Lo\`{e}ve expansion given in (\ref{klphi}) satisfies that $\alpha^2$, $\xi_1$ and $(\xi_2,\ldots,\xi_{N-1})$ are absolutely continuous and independent random vectors, $N\geq3$. Suppose that the density $f_{\xi_1}$ is Lipschitz on $\mathbb{R}$ and $\alpha^2(\omega)\in \mathcal{D}:=\cap_{N=1}^\infty \mathcal{D}_N$ for a.e. $\omega\in\Omega$. Assume that $\sum_{n=1}^\infty \|\e^{-n^2\pi^2\alpha^2 t}\|_{\leb^2(\Omega)}<\infty$ and $|T_N(\phi_1)(x,t,\alpha^2(\omega))|\geq C(x,t)>0$ for a.e. $\omega\in\Omega$ and $N$, where $\phi_1$ is the first eigenfunction in the Karhunen-Lo\`{e}ve expansion (\ref{klphi}) of $\phi$. Then the density of $u_{N,N}(x,t)(\omega)$ converges in $\leb^\infty(K)$ for every bounded set $K\subseteq\mathbb{R}$, to a density of the random variable $u(x,t)(\omega)$ given in (\ref{sol2}), for $0<x<1$ and $t>0$.
\end{theorem}
\begin{proof}
The hypothesis $\alpha^2(\omega)\in \mathcal{D}$ allows us to have the same domain of integration in (\ref{fnn}) for all $N$:
\footnotesize
\begin{align*}
{} & f_{u_{N,N}(x,t)}(u) \\
= & \int_{\mathbb{R}^{N-2}\times\mathcal{D}}f_{(\xi_1,\ldots,\xi_{N-1},\alpha^2)}\bigg(\frac{\frac{u}{2}-T_N(\mu_{\phi})(x,t,\alpha^2)-\sum_{m=2}^{N-1}T_N(\phi_m)(x,t,\alpha^2)\sqrt{\nu_m}\,\xi_m}{\sqrt{\nu_1}\,T_N(\phi_1)(x,t,\alpha^2)},\xi_2,\ldots,\xi_{N-1},\alpha^2 \bigg) \\
\cdot & \frac{1}{2\sqrt{\nu_1}\,|T_N(\phi_1)(x,t,\alpha^2)|}\,\dif \xi_2\cdots \dif \xi_{N-1}\,\dif \alpha^2. 
\end{align*}
\normalsize

Let us check that $\{f_{u_{N,N}(x,t)}\}_{n=1}^\infty$ is Cauchy in $\leb^\infty(K)$ for every bounded set $K\subseteq\mathbb{R}$, for $0<x<1$ and $t>0$. Fix two indexes $N>M$. Applying the independence between $\alpha^2$, $\xi_1$ and $(\xi_2,\ldots,\xi_{N-1})$, one gets
\footnotesize
\begin{align*}
{} & |f_{u_{N,N}(x,t)}(u)-f_{u_{M,M}(x,t)}(u)| \\
\leq & \int_{\mathbb{R}^{N-2}\times \mathcal{D}}\bigg\{\,\bigg|f_{\xi_1}\bigg(\frac{\frac{u}{2}-T_N(\mu_{\phi})(x,t,\alpha^2)-\sum_{m=2}^{N-1}T_N(\phi_m)(x,t,\alpha^2)\sqrt{\nu_m}\,\xi_m}{\sqrt{\nu_1}\,T_N(\phi_1)(x,t,\alpha^2)}\bigg)\frac{1}{2\sqrt{\nu_1}\,|T_N(\phi_1)(x,t,\alpha^2)|} \\
- & f_{\xi_1}\bigg(\frac{\frac{u}{2}-T_M(\mu_{\phi})(x,t,\alpha^2)-\sum_{m=2}^{M-1}T_M(\phi_m)(x,t,\alpha^2)\sqrt{\nu_m}\,\xi_m}{\sqrt{\nu_1}\,T_M(\phi_1)(x,t,\alpha^2)}\bigg)\frac{1}{2\sqrt{\nu_1}\,|T_M(\phi_1)(x,t,\alpha^2)|}\bigg| \\
\cdot & f_{(\xi_2,\ldots,\xi_{N-1})}(\xi_2,\ldots,\xi_{N-1})f_{\alpha^2}(\alpha^2)\,\bigg\}\,\dif \xi_2\cdots \dif \xi_{N-1}\,\dif \alpha^2 \\
\leq & \int_{\mathbb{R}^{N-2}\times \mathcal{D}} \bigg\{\,f_{\xi_1}\bigg(\frac{\frac{u}{2}-T_N(\mu_{\phi})(x,t,\alpha^2)-\sum_{m=2}^{N-1}T_N(\phi_m)(x,t,\alpha^2)\sqrt{\nu_m}\,\xi_m}{\sqrt{\nu_1}\,T_N(\phi_1)(x,t,\alpha^2)}\bigg) \\
\cdot & \bigg|\frac{1}{2\sqrt{\nu_1}\,|T_N(\phi_1)(x,t,\alpha^2)|}-\frac{1}{2\sqrt{\nu_1}\,|T_M(\phi_1)(x,t,\alpha^2)|}\bigg|  \\
\cdot & f_{(\xi_2,\ldots,\xi_{N-1})}(\xi_2,\ldots,\xi_{N-1})f_{\alpha^2}(\alpha^2)\,\bigg\}\,\dif \xi_2\cdots \dif \xi_{N-1}\,\dif \alpha^2  \\
+ & \int_{\mathbb{R}^{N-2}\times \mathcal{D}} \bigg\{\,\frac{1}{2\sqrt{\nu_1}\,|T_M(\phi_1)(x,t,\alpha^2)|}  \\
\cdot & \bigg| f_{\xi_1}\bigg(\frac{\frac{u}{2}-T_N(\mu_{\phi})(x,t,\alpha^2)-\sum_{m=2}^{N-1}T_N(\phi_m)(x,t,\alpha^2)\sqrt{\nu_m}\,\xi_m}{\sqrt{\nu_1}\,T_N(\phi_1)(x,t,\alpha^2)}\bigg) \\
- & f_{\xi_1}\bigg(\frac{\frac{u}{2}-T_M(\mu_{\phi})(x,t,\alpha^2)-\sum_{m=2}^{M-1}T_M(\phi_m)(x,t,\alpha^2)\sqrt{\nu_m}\,\xi_m}{\sqrt{\nu_1}\,T_M(\phi_1)(x,t,\alpha^2)}\bigg) \bigg| \\
\cdot & f_{(\xi_2,\ldots,\xi_{N-1})}(\xi_2,\ldots,\xi_{N-1})f_{\alpha^2}(\alpha^2)\,\bigg\}\,\dif \xi_2\cdots \dif \xi_{N-1}\,\dif \alpha^2 \stackrel{\Delta}{=}(I_1)+(I_2).
\end{align*}
\normalsize

Call $L$ the Lipschitz constant of $f_{\xi_1}$. Denote by $F_{1,0}=f_{\xi_1}(0)$ and by $\phi^{(N)}(x)=\mu_\phi(x)+\sum_{n=1}^N \sqrt{\nu_n}\,\phi_n(x)\,\xi_n$
the $N$-th partial sum of the Karhunen-Lo\`{e}ve expansion (\ref{klphi}).

We carry out four inequalities that will appear when we bound $(I_1)$ and $(I_2)$: for a general $f\in \leb^2(0,1)$,
\small
\begin{equation}
 |T_N(f)(x,t,\alpha^2)|\leq \sum_{n=1}^N|\hat{f}(n)|\e^{-n^2\pi^2\alpha^2 t}\leq \|f\|_{\leb^1(0,1)}\sum_{n=1}^N \e^{-n^2\pi^2\alpha^2 t}, 
\label{1a}
\end{equation}
\begin{equation}
|T_N(f)(x,t,\alpha^2)-T_M(f)(x,t,\alpha^2)|\leq \sum_{n=M+1}^N|\hat{f}(n)|\e^{-n^2\pi^2\alpha^2 t}\leq \|f\|_{\leb^1(0,1)}\sum_{n=M+1}^N \e^{-n^2\pi^2\alpha^2 t}, 
\label{2a}
\end{equation}
\normalsize
\footnotesize
\begin{align}
{} & \left| \sum_{m=2}^{N-1} T_N(\phi_m)(x,t,\alpha^2)\sqrt{\nu_m}\,\xi_m\right|=\left|\sum_{m=2}^{N-1}\sum_{n=1}^N \left(\int_0^1 \phi_m(y)\sin(n\pi y)\,\dif y\right)\e^{-n^2\pi^2\alpha^2 t}\sin(n\pi x)\sqrt{\nu_m}\,\xi_m\right| \nonumber \\
\leq & \sum_{n=1}^N \int_0^1 \left|\sum_{m=2}^{N-1}\phi_m(y)\sqrt{\nu_m}\,\xi_m\right|\,\dif y\,\e^{-n^2\pi^2\alpha^2 t}=\left(\int_0^1 \left|\phi^{(N-1)}(y)-\phi^{(1)}(y)\right|\,\dif y\right)\sum_{n=1}^N \e^{-n^2\pi^2\alpha^2 t} \label{3a}
\end{align}
\normalsize
and
\small
\begin{align}
{} & \left| \sum_{m=2}^{N-1} T_N(\phi_m)(x,t,\alpha^2)\sqrt{\nu_m}\,\xi_m - \sum_{m=2}^{M-1} T_M(\phi_m)(x,t,\alpha^2)\sqrt{\nu_m}\,\xi_m\right| \nonumber \\
\leq & \bigg|\sum_{m=2}^{N-1}\sum_{n=1}^N \left(\int_0^1 \phi_m(y)\sin(n\pi y)\,\dif y\right)\e^{-n^2\pi^2\alpha^2 t}\sin(n\pi x)\sqrt{\nu_m}\,\xi_m \nonumber \\
- & \sum_{m=2}^{M-1}\sum_{n=1}^M \left(\int_0^1 \phi_m(y)\sin(n\pi y)\,\dif y\right)\e^{-n^2\pi^2\alpha^2 t}\sin(n\pi x)\sqrt{\nu_m}\,\xi_m\bigg| \;\;(\text{add and subtract }\sum_{m=2}^{N-1}\sum_{n=1}^M) \nonumber \\
\leq & \left|\sum_{m=2}^{N-1}\sum_{n=M+1}^N \left(\int_0^1 \phi_m(y)\sin(n\pi y)\,\dif y\right)\e^{-n^2\pi^2\alpha^2 t}\sin(n\pi x)\sqrt{\nu_m}\,\xi_m\right| \nonumber \\
+ & \left|\sum_{n=1}^M \sum_{m=M}^{N-1} \left(\int_0^1 \phi_m(y)\sin(n\pi y)\,\dif y\right)\e^{-n^2\pi^2\alpha^2 t}\sin(n\pi x)\sqrt{\nu_m}\,\xi_m\right| \nonumber \\
\leq & \sum_{n=M+1}^N \int_0^1 \left|\sum_{m=2}^{N-1}\phi_m(y)\sqrt{\nu_m}\,\xi_m\right|\,\dif y\,\e^{-n^2\pi^2\alpha^2 t} \nonumber \\
+ & \sum_{n=1}^M \int_0^1 \left|\sum_{m=M}^{N-1}\phi_m(y)\sqrt{\nu_m}\,\xi_m\right|\,\dif y\,\e^{-n^2\pi^2\alpha^2 t} \nonumber \\
= & \left(\int_0^1 \left|\phi^{(N-1)}(y)-\phi^{(1)}(y)\right|\,\dif y\right)\sum_{n=M+1}^N \e^{-n^2\pi^2\alpha^2 t} \nonumber \\
+ & \left(\int_0^1 \left|\phi^{(N-1)}(y)-\phi^{(M-1)}(y)\right|\,\dif y\right)\sum_{n=1}^M \e^{-n^2\pi^2\alpha^2 t}. \label{4a}
\end{align}
\normalsize

From now on in this proof, $C$ will denote any constant whose value depends on $x$, $t$ and $\phi$, and it does not depend on $N$, $M$ and $u$. The reason is that the expressions to deal with will become large and we do not want the notation to be cumbersome.

Let us bound $(I_1)$. First we apply the Lipschitz condition of $f_{\xi_1}$ and bounds (\ref{1a}) and (\ref{3a}):
\small
\begin{align*}
{} & f_{\xi_1}\bigg(\frac{\frac{u}{2}-T_N(\mu_{\phi})(x,t,\alpha^2)-\sum_{m=2}^{N-1}T_N(\phi_m)(x,t,\alpha^2)\sqrt{\nu_m}\,\xi_m}{\sqrt{\nu_1}\,T_N(\phi_1)(x,t,\alpha^2)}\bigg) \\
\leq & L\bigg|\frac{\frac{u}{2}-T_N(\mu_{\phi})(x,t,\alpha^2)-\sum_{m=2}^{N-1}T_N(\phi_m)(x,t,\alpha^2)\sqrt{\nu_m}\,\xi_m}{\sqrt{\nu_1}\,T_N(\phi_1)(x,t,\alpha^2)}\bigg|+F_{1,0} \\
\leq & \frac{L}{\sqrt{\nu_1}\,C(x,t)}\left(\frac{|u|}{2}+|T_N(\mu_{\phi})(x,t,\alpha^2)|+\left|\sum_{m=2}^{N-1}T_N(\phi_m)(x,t,\alpha^2)\sqrt{\nu_m}\,\xi_m\right|\right)+F_{1,0} \\
\leq & C\left(|u|+\sum_{n=1}^N \e^{-n^2\pi^2 \alpha^2 t}+\left(\int_0^1|\phi^{(N-1)}(y)-\phi^{(1)}(y)|\,\dif y\right)\sum_{n=1}^N \e^{-n^2\pi^2 \alpha^2 t}+1\right).
\end{align*}
\normalsize
Using bound (\ref{2a}),
\small
\begin{align*}
{} & \bigg|\frac{1}{2\sqrt{\nu_1}\,|T_N(\phi_1)(x,t,\alpha^2)|}-\frac{1}{2\sqrt{\nu_1}\,|T_M(\phi_1)(x,t,\alpha^2)|}\bigg| \\
= & \frac{\big|\,|T_N(\phi_1)(x,t,\alpha^2)|-|T_M(\phi_1)(x,t,\alpha^2)|\,\big|}{2\sqrt{\nu_1}\,|T_N(\phi_1)(x,t,\alpha^2)|\,|T_M(\phi_1)(x,t,\alpha^2)|} \leq\frac{|T_N(\phi_1)(x,t,\alpha^2)-T_M(\phi_1)(x,t,\alpha^2)|}{2\sqrt{\nu_1}\,|T_N(\phi_1)(x,t,\alpha^2)|\,|T_M(\phi_1)(x,t,\alpha^2)|} \\
\leq & C|T_N(\phi_1)(x,t,\alpha^2)-T_M(\phi_1)(x,t,\alpha^2)|\leq C\sum_{n=M+1}^N \e^{-n^2\pi^2\alpha^2 t}.
\end{align*}
\normalsize
This implies
\small
\begin{align*}
{} & (I_1) \\
\leq & C\,\mathbb{E}\bigg[\left(|u|+\sum_{n=1}^N \e^{-n^2\pi^2 \alpha^2 t}+\left(\int_0^1|\phi^{(N-1)}(y)-\phi^{(1)}(y)|\,\dif y\right)\sum_{n=1}^N \e^{-n^2\pi^2 \alpha^2 t}+1\right) \\
\cdot & \left(\sum_{n=M+1}^N \e^{-n^2\pi^2\alpha^2 t}\right)\bigg]\;\;(\text{expand, use linearity of }\mathbb{E}\text{ and independence of }\phi^{(j)}\text{ and }\alpha^2,\,j\geq1) \\
= & C\bigg\{(|u|+1)\sum_{n=M+1}^N \mathbb{E}[\e^{-n^2\pi^2\alpha^2 t}]+\sum_{n=1}^N\sum_{m=M+1}^N \mathbb{E}[\e^{-n^2\pi^2\alpha^2 t}\e^{-m^2\pi^2\alpha^2 t}] \\
+ & \|\phi^{(N-1)}-\phi^{(1)}\|_{\leb^1([0,1]\times\Omega)}\sum_{n=1}^N\sum_{m=M+1}^N \mathbb{E}[\e^{-n^2\pi^2\alpha^2 t}\e^{-m^2\pi^2\alpha^2 t}]\bigg\}\;\;(\text{use Cauchy-Schwarz}) \\
\leq & C\bigg\{(|u|+1)\sum_{n=M+1}^N \|\e^{-n^2\pi^2\alpha^2 t}\|_{\leb^2(\Omega)}+\sum_{n=1}^N\sum_{m=M+1}^N \|\e^{-n^2\pi^2\alpha^2 t}\|_{\leb^2(\Omega)}\|\e^{-m^2\pi^2\alpha^2 t}\|_{\leb^2(\Omega)} \\
+ & \|\phi^{(N-1)}-\phi^{(1)}\|_{\leb^2([0,1]\times\Omega)}\sum_{n=1}^N\sum_{m=M+1}^N \|\e^{-n^2\pi^2\alpha^2 t}\|_{\leb^2(\Omega)}\|\e^{-m^2\pi^2\alpha^2 t}\|_{\leb^2(\Omega)}\bigg\}\;\;(\text{group terms}) \\
= & C\bigg\{(|u|+1)\sum_{n=M+1}^N\|\e^{-n^2\pi^2\alpha^2 t}\|_{\leb^2(\Omega)}+\left(\sum_{n=1}^N\|\e^{-n^2\pi^2\alpha^2 t}\|_{\leb^2(\Omega)}\right)\left(\sum_{n=M+1}^N\|\e^{-n^2\pi^2\alpha^2 t}\|_{\leb^2(\Omega)}\right) \\
\cdot & \left(1+\|\phi^{(N-1)}-\phi^{(1)}\|_{\leb^2([0,1]\times\Omega)}\right)\bigg\}.
\end{align*}
\normalsize

Let us bound $(I_2)$. First,
\[ \frac{1}{2\sqrt{\nu_1}\,|T_M(\phi_1)(x,t,\alpha^2)|}\leq C. \]
Now, using the Lipschitz condition of $f_{\xi_1}$ and inequality (\ref{4a}),
\small
\begin{align*}
{} & \bigg| f_{\xi_1}\bigg(\frac{\frac{u}{2}-T_N(\mu_{\phi})(x,t,\alpha^2)-\sum_{m=2}^{N-1}T_N(\phi_m)(x,t,\alpha^2)\sqrt{\nu_m}\,\xi_m}{\sqrt{\nu_1}\,T_N(\phi_1)(x,t,\alpha^2)}\bigg) \\
- & f_{\xi_1}\bigg(\frac{\frac{u}{2}-T_M(\mu_{\phi})(x,t,\alpha^2)-\sum_{m=2}^{M-1}T_M(\phi_m)(x,t,\alpha^2)\sqrt{\nu_m}\,\xi_m}{\sqrt{\nu_1}\,T_M(\phi_1)(x,t,\alpha^2)}\bigg) \bigg| \\
\leq & L\,\bigg|\frac{\frac{u}{2}-T_N(\mu_{\phi})(x,t,\alpha^2)-\sum_{m=2}^{N-1}T_N(\phi_m)(x,t,\alpha^2)\sqrt{\nu_m}\,\xi_m}{\sqrt{\nu_1}\,T_N(\phi_1)(x,t,\alpha^2)} \\
- & \frac{\frac{u}{2}-T_M(\mu_{\phi})(x,t,\alpha^2)- \sum_{m=2}^{M-1}T_M(\phi_m)(x,t,\alpha^2)\sqrt{\nu_m}\,\xi_m}{\sqrt{\nu_1}\,T_M(\phi_1)(x,t,\alpha^2)}\bigg| \\
\leq & L\bigg\{ \bigg|\frac{1}{\sqrt{\nu_1}\,T_N(\phi_1)(x,t,\alpha^2)}-\frac{1}{\sqrt{\nu_1}\,T_M(\phi_1)(x,t,\alpha^2)}\bigg| \\
\cdot & \bigg|\frac{u}{2}-T_N(\mu_{\phi})(x,t,\alpha^2)-\sum_{m=2}^{N-1}T_N(\phi_m)(x,t,\alpha^2)\sqrt{\nu_m}\,\xi_m\bigg| \\
+ & \frac{1}{\sqrt{\nu_1}\,|T_M(\phi_1)(x,t,\alpha^2)|}\bigg(|T_N(\mu_{\phi})(x,t,\alpha^2)-T_M(\mu_{\phi})(x,t,\alpha^2)| \\
+ & \bigg|\sum_{m=2}^{N-1}T_N(\phi_m)(x,t,\alpha^2)\sqrt{\nu_m}\,\xi_m-\sum_{m=2}^{M-1}T_M(\phi_m)(x,t,\alpha^2)\sqrt{\nu_m}\,\xi_m\bigg|\bigg)\bigg\} \\
\leq & C\bigg\{ \left(\sum_{n=M+1}^N \e^{-n^2\pi^2\alpha^2 t}\right)\left(|u|+\sum_{n=1}^N \e^{-n^2\pi^2\alpha^2 t}+\left(\int_0^1 |\phi^{(N-1)}(y)-\phi^{(1)}(y)|\,\dif y\right)\sum_{n=1}^N \e^{-n^2\pi^2\alpha^2 t}\right) \\
+ & \sum_{n=M+1}^N \e^{-n^2\pi^2\alpha^2 t}+\left(\int_0^1 |\phi^{(N-1)}(y)-\phi^{(1)}(y)|\,\dif y\right)\sum_{n=M+1}^N \e^{-n^2\pi^2\alpha^2 t} \\
+ & \left(\int_0^1 |\phi^{(N-1)}(y)-\phi^{(M-1)}(y)|\,\dif y\right)\sum_{n=1}^M \e^{-n^2\pi^2\alpha^2 t}\bigg\}.
\end{align*}
\normalsize
These inequalities give
\footnotesize
\begin{align*}
{} & (I_2) \\
\leq & C\,\mathbb{E}\bigg[ \left(\sum_{n=M+1}^N \e^{-n^2\pi^2\alpha^2 t}\right)\left(|u|+\sum_{n=1}^N \e^{-n^2\pi^2\alpha^2 t}+\left(\int_0^1 |\phi^{(N-1)}(y)-\phi^{(1)}(y)|\,\dif y\right)\sum_{n=1}^N \e^{-n^2\pi^2\alpha^2 t}\right) \\
+ & \sum_{n=M+1}^N \e^{-n^2\pi^2\alpha^2 t}+\left(\int_0^1 |\phi^{(N-1)}(y)-\phi^{(1)}(y)|\,\dif y\right)\sum_{n=M+1}^N \e^{-n^2\pi^2\alpha^2 t} \\
+ & \left(\int_0^1 |\phi^{(N-1)}(y)-\phi^{(M-1)}(y)|\,\dif y\right)\sum_{n=1}^M \e^{-n^2\pi^2\alpha^2 t}\bigg] \\
& (\text{expand, use linearity of }\mathbb{E}\text{ and independence of }\phi^{(j)}\text{ and }\alpha^2,\,j\geq1) \\
= & C\bigg\{ |u|\sum_{n=M+1}^N \mathbb{E}[\e^{-n^2\pi^2\alpha^2 t}]+\sum_{n=1}^M \sum_{m=M+1}^N \mathbb{E}[\e^{-n^2\pi^2\alpha^2 t}\e^{-m^2\pi^2\alpha^2 t}] \\
+ & \|\phi^{(N-1)}-\phi^{(1)}\|_{\leb^1([0,1]\times\Omega)}\sum_{n=1}^M \sum_{m=M+1}^N \mathbb{E}[\e^{-n^2\pi^2\alpha^2 t}\e^{-m^2\pi^2\alpha^2 t}]+\sum_{n=M+1}^N \mathbb{E}[\e^{-n^2\pi^2\alpha^2 t}] \\
+ & \|\phi^{(N-1)}-\phi^{(1)}\|_{\leb^1([0,1]\times\Omega)}\sum_{n=M+1}^N \mathbb{E}[\e^{-n^2\pi^2\alpha^2 t}]+ \|\phi^{(N-1)}-\phi^{(M-1)}\|_{\leb^1([0,1]\times\Omega)}\sum_{n=1}^M \mathbb{E}[\e^{-n^2\pi^2\alpha^2 t}]\bigg\} \\
& (\text{use Cauchy-Schwarz}) \\ 
\leq & C\bigg\{ |u|\sum_{n=M+1}^N \|\e^{-n^2\pi^2\alpha^2 t}\|_{\leb^2(\Omega)}+\sum_{n=1}^M \sum_{m=M+1}^N \|\e^{-n^2\pi^2\alpha^2 t}\|_{\leb^2(\Omega)}\|\e^{-m^2\pi^2\alpha^2 t}\|_{\leb^2(\Omega)} \\
+ & \|\phi^{(N-1)}-\phi^{(1)}\|_{\leb^2([0,1]\times\Omega)}\sum_{n=1}^M \sum_{m=M+1}^N \|\e^{-n^2\pi^2\alpha^2 t}\|_{\leb^2(\Omega)}\|\e^{-m^2\pi^2\alpha^2 t}\|_{\leb^2(\Omega)}+\sum_{n=M+1}^N \|\e^{-n^2\pi^2\alpha^2 t}\|_{\leb^2(\Omega)} \\
+ & \|\phi^{(N-1)}-\phi^{(1)}\|_{\leb^2([0,1]\times\Omega)}\sum_{n=M+1}^N \|\e^{-n^2\pi^2\alpha^2 t}\|_{\leb^2(\Omega)}+ \|\phi^{(N-1)}-\phi^{(M-1)}\|_{\leb^2([0,1]\times\Omega)}\sum_{n=1}^M \|\e^{-n^2\pi^2\alpha^2 t}\|_{\leb^2(\Omega)}\bigg\} \\
& (\text{group terms}) \\
\leq & C\bigg\{ (|u|+1)\sum_{n=M+1}^N \| \e^{-n^2\pi^2\alpha^2 t}\|_{\leb^2(\Omega)} \\
+ & \left(\sum_{n=1}^N \| \e^{-n^2\pi^2\alpha^2 t}\|_{\leb^2(\Omega)}\right)\left(\sum_{n=M+1}^N \| \e^{-n^2\pi^2\alpha^2 t}\|_{\leb^2(\Omega)}\right)\left(\|\phi^{(N-1)}-\phi^{(1)}\|_{\leb^2([0,1]\times\Omega)}+1\right) \\
+ & \|\phi^{(N-1)}-\phi^{(1)}\|_{\leb^2([0,1]\times\Omega)}\sum_{n=M+1}^N \| \e^{-n^2\pi^2\alpha^2 t}\|_{\leb^2(\Omega)}+\|\phi^{(N-1)}-\phi^{(M-1)}\|_{\leb^2([0,1]\times\Omega)}\sum_{n=1}^M \| \e^{-n^2\pi^2\alpha^2 t}\|_{\leb^2(\Omega)} \bigg\}.
\end{align*}

\normalsize

From the hypotheses $\sum_{n=1}^\infty \|\e^{-n^2\pi^2\alpha^2 t}\|_{\leb^2(\Omega)}<\infty$ and $\phi^{(N)}\stackrel{N\rightarrow\infty}{\longrightarrow}\phi$ in $\leb^2([0,1]\times\Omega)$, we arrive at the desired result: $\{f_{u_{N,N}(x,t)}\}_{n=1}^\infty$ is Cauchy in $\leb^\infty(K)$ for every bounded set $K\subseteq\mathbb{R}$, for $0<x<1$ and $t>0$.

Let $g_{x,t}(u)=\lim_{N\rightarrow\infty} f_{u_{N,N}(x,t)}(u)$, $u\in\mathbb{R}$. We need to check that $g_{x,t}$ is a density $f_{u(x,t)}$ of the random variable $u(x,t)(\omega)$ given in (\ref{sol2}), for $0<x<1$ and $t>0$. As we did in the end of the proof of Theorem \ref{teor1}, it suffices to check that $u_{N,N}(x,t)$ converges in law to $u(x,t)$, so that $\lim_{N\rightarrow\infty} F_{u_{N,N}(x,t)}(u)=F_{u(x,t)}(u)$ for all $u\in\mathbb{R}$ being a point of continuity of $F_{u(x,t)}$. As we saw in the end of the proof of Theorem \ref{teor1}, this would imply that $g_{x,t}$ is a density $f_{u(x,t)}$ of the random variable $u(x,t)(\omega)$.

We show that $u_{N,N}(x,t)\rightarrow u(x,t)$ in $\leb^1(\Omega)$ as $N\rightarrow\infty$, for $0<x<1$ and $t>0$. This will imply the desired convergence in law.

Write 
\begin{align*}
 u(x,t)(\omega)= {} & 2\sum_{n=1}^\infty \hat{\mu}_\phi(n)\e^{-n^2\pi^2\alpha^2(\omega)t}\sin(n\pi x) \\
+ & 2\sum_{n=1}^\infty \int_0^1 \left(\sum_{m=1}^\infty\sqrt{\nu_m}\,\phi_m(y)\xi_m(\omega)\right)\sin(n\pi y)\,\dif y\,\e^{-n^2\pi^2\alpha^2(\omega) t}\sin(n\pi x) \\
\stackrel{\Delta}{=} & 2\cdot \mathrm{(A1)}+2\cdot \mathrm{(A2)}, 
\end{align*}
where the sum $\sum_{m=1}^\infty$ is in the topology of $\leb^2([0,1]\times\Omega)$ and both sums $\sum_{n=1}^\infty$ are understood pointwise. Write
\small
\begin{align*}
 u_{N,N}(x,t)(\omega)= {} & 2\sum_{n=1}^N \hat{\mu}_\phi(n)\e^{-n^2\pi^2\alpha^2(\omega)t}\sin(n\pi x) \\
+ & 2\sum_{n=1}^N \int_0^1 \left(\sum_{m=1}^{N-1}\sqrt{\nu_m}\,\phi_m(y)\xi_m(\omega)\right)\sin(n\pi y)\,\dif y\,\e^{-n^2\pi^2\alpha^2(\omega) t}\sin(n\pi x) \\
\stackrel{\Delta}{=} & 2\cdot \mathrm{(A3)}+2\cdot \mathrm{(A4)}. 
\end{align*}
\normalsize
Let us perform some estimates:
\begin{align*}
 \mathbb{E}[|\mathrm{(A1)}-\mathrm{(A3)}|]\leq & \|\mu_\phi\|_{\leb^1(0,1)}\sum_{n=N+1}^\infty \mathbb{E}[\e^{-n^2\pi^2\alpha^2 t}] \\
\leq & \|\mu_\phi\|_{\leb^1(0,1)}\sum_{n=N+1}^\infty \|\e^{-n^2\pi^2\alpha^2 t}\|_{\leb^2(\Omega)} \stackrel{N\rightarrow\infty}{\longrightarrow} 0 
\end{align*}
(by Cauchy-Schwarz and the hypothesis $\sum_{n=1}^\infty \|\e^{-n^2\pi^2\alpha^2 t}\|_{\leb^2(\Omega)}<\infty$) and
\begin{align*}
{} & \mathbb{E}[|\mathrm{(A2)}-\mathrm{(A4)}|]\,\;(\text{add and subtract }\sum_{n=1}^N\sum_{m=1}^\infty) \\
\leq & \mathbb{E}\left[\left|\sum_{n=N+1}^\infty\int_0^1 (\phi(y)-\mu_\phi(y))\sin(n\pi y)\,\dif y\,\e^{-n^2\pi^2\alpha^2 t}\sin(n\pi x)\right|\right] \\
+ & \mathbb{E}\left[\left|\sum_{n=1}^N\int_0^1 (\phi(y)-\phi^{(N-1)}(y))\sin(n\pi y)\,\dif y\,\e^{-n^2\pi^2\alpha^2 t}\sin(n\pi x)\right|\right] \\
\leq & \|\phi-\mu_\phi\|_{\leb^2([0,1]\times\Omega)}\sum_{n=N+1}^\infty \|\e^{-n^2\pi^2\alpha^2 t}\|_{\leb^2(\Omega)} \\
+ & \|\phi-\phi^{(N-1)}\|_{\leb^2([0,1]\times\Omega)}\sum_{n=1}^N \|\e^{-n^2\pi^2\alpha^2 t}\|_{\leb^2(\Omega)}\stackrel{N\rightarrow\infty}{\longrightarrow} 0.
\end{align*}
This proves that $u_{N,N}(x,t)\rightarrow u(x,t)$ as $N\rightarrow\infty$ in $\leb^1(\Omega)$, for $0<x<1$ and $t>0$, and we are done.
\end{proof}

To conclude, we make some comments on the hypotheses of Theorem \ref{teor3}. 

\begin{remark}
If $\phi$ is a Gaussian process, then $\xi_1,\xi_2,\ldots$ are independent and Gaussian. Thus, $f_{\xi_1}$ is Lipschitz on $\mathbb{R}$. On the other hand, if $\alpha^2(\omega)\geq a>0$ for a.e. $\omega \in \Omega$, then the hypothesis $\sum_{n=1}^{\infty} \|\e^{-n^2\pi^2\alpha^2 t}\|_{\leb^2(\Omega)}<\infty$ holds.
\end{remark}

\begin{remark}
The hypothesis $|T_N(\phi_1)(x,t,\alpha^2(\omega))|\geq C(x,t)>0$ for a.e. $\omega\in\Omega$ and $N$, is very difficult to check in practice. 

For example, if $\phi(x)=W(x)$, where $W$ is a standard Brownian motion on $[0,1]$, then the eigenvalues and eigenvectors associated to its Karhunen-Lo\`{e}ve expansion are 
\[\nu_{j+1}=\frac{1}{\left(j+\frac12\right)^2\pi^2},\quad \phi_{j+1}(t)=\sqrt{2}\sin\left(t\left(j+\frac12\right)\pi\right),\quad j\geq 0.\]
We have
\[\hat{\phi}_1(n)=\int_0^1\phi_1(y)\sin(n\pi y)\,\dif y=\sqrt{2}\int_0^1\sin\left(y\frac{\pi}{2}\right)\sin(n\pi y)\,\dif y=\sqrt{2}\frac{4n(-1)^n}{\pi(1-4n^2)},\]
therefore,
\[T_N(\phi_1)(x,t,\alpha^2(\omega))=\frac{4\sqrt{2}}{\pi}\sum_{n=1}^N\frac{n(-1)^n}{1-4n^2}\e^{-n^2\pi^2\alpha^2(\omega)t}\sin(n\pi x).\]
In principle, for a given $\alpha^2$ it is not possible to ensure directly that $|T_N(\phi_1)(x,t,\alpha^2(\omega))|\geq C(x,t)>0$ for a.e. $\omega\in\Omega$ and $N$ happens. In fact, this hypothesis cannot hold, since as we said in Remark \ref{brownia_no_funciona}, the initial condition $\phi$ cannot be a Brownian motion, because $\phi(1)=u(1,0)=0$.

However, if $\phi(x)=B(x)$, where $B$ is a Brownian bridge on $[0,1]$, it is possible to ensure that the hypothesis holds. Indeed, the eigenvalues and eigenvectors associated to its Karhunen-Lo\`{e}ve expansion are 
\[\nu_j=\frac{1}{\pi^2 j^2}, \quad \phi_j(t)=\sqrt{2}\sin(j\pi t),\quad j\geq1. \]
The key fact is that the orthonormal system of $\leb^2([0,1])$ obtained in the Karhunen-Lo\`{e}ve expansion of the Brownian bridge coincides with the orthonormal system of $\leb^2([0,1])$ obtained in the Sturm-Liouville problem associated to the PDE problem (\ref{edp_determinista}). Concerning computations, this implies that
\[\hat{\phi}_1(n)=\sqrt{2}\int_0^1\sin(\pi y)\sin(n\pi y)\,\dif y=\begin{cases} 0,&\; n\neq1,\\ \frac{\sqrt{2}}{2},&\; n=1.\end{cases}\]
Thus, \[T_N(\phi_1)(x,t,\alpha^2(\omega))=\frac{\sqrt{2}}{2}\e^{-\pi^2 \alpha^2(\omega)t}\sin(\pi x).\]
If $\support (\alpha^2)\subseteq[a,b]\subseteq(0,\infty)$, then 
\[ T_N(\phi_1)(x,t,\alpha^2(\omega))\geq\frac{\sqrt{2}}{2}\e^{-\pi^2 b\,t}\sin(\pi x)=:C(x,t)>0,  \]
for $0<x<1$ and $t>0$. Hence, the hypothesis holds for the Brownian bridge.

As we commented in Remark \ref{bridge_funciona}, it makes sense to model the initial condition by means of a Brownian bridge, since at $x=1$ the process $\phi$ must vanish.
\end{remark}

\section{Examples}

\begin{example} \label{ex_bb} \normalfont
Consider the randomized PDE problem (\ref{edp_determinista}), with $\alpha^2\sim \text{Uniform}(1,2)$ and $\phi(x)=B(x)$ a standard Brownian bridge on $[0,1]$ being independent. Recall that the hypotheses of Theorem \ref{teor2} and Theorem \ref{teor3} are satisfied. We will perform numerical approximations of the probability density function of the solution $u(x,t)(\omega)$ given in (\ref{sol2}). For that purpose, we will use formulas (\ref{fbb1}) and (\ref{fnn}), which give $f_{u_N(x,t)}(u)$ and $f_{u_{N,N}(x,t)}(u)$ respectively.

In Figures \ref{x05t01}, \ref{x07t03} and \ref{x07t1}, we can see the density $f_{u_N(x,t)}(u)$ given in (\ref{fbb1}) for $N=2$ (left) and $N=3$ (right) at the points $(x,t)=(0.5,0.1)$, $(x,t)=(0.7,0.3)$ and $(x,t)=(0.7,1)$, respectively. In Figures \ref{x05t01KL}, \ref{x07t03KL} and \ref{x07t1KL}, we can see the density $f_{u_{N,N}(x,t)}(u)$ given in (\ref{fnn}) for $N=3$ (left) and $N=4$ (right) at the same points as before. In Figure \ref{3D}, three dimensional plots of the density $f_{u_3(x,t)}(u)$ given in (\ref{fbb1}) (left) and of the density $f_{u_{4,4}(x,t)}(u)$ given in (\ref{fnn}) (right) are presented, with $x=0.5$ fixed and $t\in [0.1,0.5]$ varying, to show the time evolution of the density.

In Table \ref{taulaL1}, we compare the two plots in each of the figures in order to assess convergence. In Table \ref{taulaEV}, we simulate the expectation and variance of $u(x,t)(\omega)$ at the previous points.

Notice that, as $t$ increases, the density of $u(x,t)(\omega)$ seems to behave as a Dirac delta function. Indeed, as $A_1,A_2,\ldots$ are independent and $A_n\sim \text{Normal}(0,2/(n^2\pi^2))$ by Remark \ref{bridge_funciona}, we have
\[\mathbb{E}[u(x,t)]=\sum_{n=1}^\infty \mathbb{E}[A_n]  \mathbb{E}[\e^{-n^2\pi^2\alpha^2t}]\sin(n\pi x)=0 \]
since $\mathbb{E}[A_n]=0$ for all $n=1,2,\ldots$ and, taking into account that $\alpha^2(\omega)\geq1$ for a.e. $\omega\in\Omega$,
\begin{align*}
\mathbb{V}[u(x,t)]= {} & \|u(x,t)\|_{\leb^2(\Omega)}^2 \leq \left\|\sum_{n=1}^{\infty} |A_n|\,\e^{-n^2\pi^2t}\right\|_{\leb^2(\Omega)}^2=\sum_{n=1}^\infty \|A_n\|_{\leb^2(\Omega)}^2\,\e^{-2n^2\pi^2 t} \\
= {} & \sum_{n=1}^\infty \frac{2}{n^2\pi^2} \e^{-2n^2\pi^2t} \stackrel{t\rightarrow\infty}{\longrightarrow} 0.
\end{align*}
Therefore, the density tends to be concentrated around zero. 

\bigskip

\begin{table}[H]
\begin{center}
\begin{tabular}{|c|c|c|c|} \hline
$\leb^1$ / $(x,t)$ & $(0.5,0.1)$ & $(0.7,0.3)$ & $(0.7,1)$  \\ \hline
$\|f_{u_2(x,t)}-f_{u_3(x,t)}\|_{\leb^1(\mathbb{R})}$ & $1.65393\cdot 10^{-8}$ & $2.51309\cdot 10^{-7}$ & $0.00734303$ \\ \hline
$\|f_{u_{3,3}(x,t)}-f_{u_{4,4}(x,t)}\|_{\leb^1(\mathbb{R})}$ & $5.60959\cdot 10^{-8}$ & $1.14085\cdot 10^{-7}$ & $0.000148152$ \\ \hline
\end{tabular}
\caption{Comparison of the two plots in each of the figures. Example \ref{ex_bb}.}
\label{taulaL1}
\end{center}
\end{table}

\begin{table}[H]
\begin{center}
\begin{tabular}{|c|c|c|c|} \hline
$\mathbb{E}$, $\mathbb{V}$ / $(x,t)$ & $(0.5,0.1)$ & $(0.7,0.3)$ & $(0.7,1)$  \\ \hline
$\mathbb{E}[u_3(x,t)]$ & $2.14238 \cdot 10^{-18}$ & $3.00816\cdot 10^{-18}$ & $0$ \\ \hline
$\mathbb{V}[u_3(x,t)]$ & $0.0122708$ & $0.0000593315$ & $0$ \\ \hline
$\mathbb{E}[u_{4,4}(x,t)]$ & $-1.58646\cdot 10^{-17}$ & $2.32603\cdot 10^{-18}$ & $0$ \\ \hline
$\mathbb{V}[u_{4,4}(x,t)]$ & $0.0122708$ & $0.0000593237$ & $0$ \\ \hline
\end{tabular}
\caption{Simulation of the expectation and variance. Example \ref{ex_bb}.}
\label{taulaEV}
\end{center}
\end{table} 

\begin{figure}[H]
  \begin{center}
    \includegraphics[width=7cm]{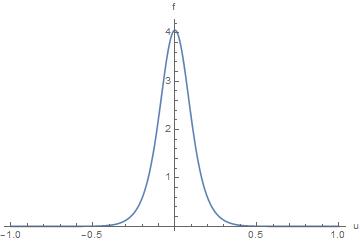}
		\includegraphics[width=7cm]{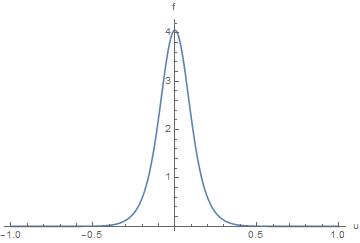}
    \caption{Density (\ref{fbb1}) for $N=2$ (left) and $N=3$ (right) at the point $(x,t)=(0.5,0.1)$. Example \ref{ex_bb}.}
		\label{x05t01}
    \end{center}
  \end{figure}
	
\begin{figure}[H]
  \begin{center}
    \includegraphics[width=7cm]{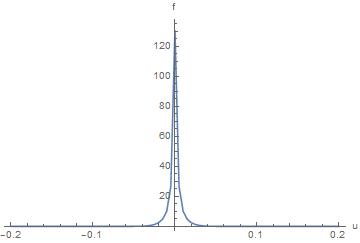}
		\includegraphics[width=7cm]{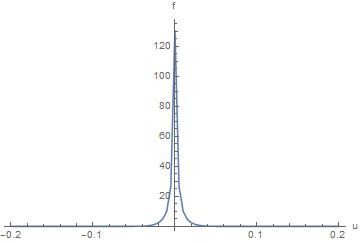}
    \caption{Density (\ref{fbb1}) for $N=2$ (left) and $N=3$ (right) at the point $(x,t)=(0.7,0.3)$. Example \ref{ex_bb}.}
		\label{x07t03}
    \end{center}
  \end{figure}
	
	\begin{figure}[H]
  \begin{center}
    \includegraphics[width=7cm]{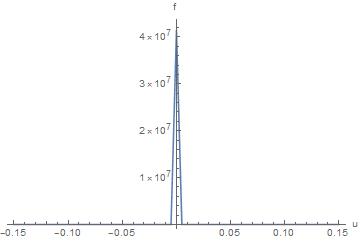}
		\includegraphics[width=7cm]{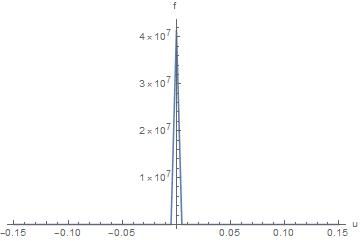}
    \caption{Density (\ref{fbb1}) for $N=2$ (left) and $N=3$ (right) at the point $(x,t)=(0.7,1)$. Example \ref{ex_bb}.}
		\label{x07t1}
    \end{center}
  \end{figure}

	\begin{figure}[H]
  \begin{center}
    \includegraphics[width=7cm]{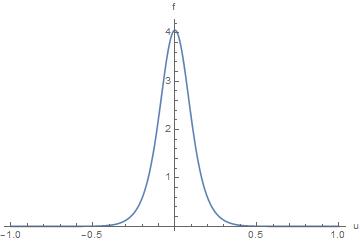}
		\includegraphics[width=7cm]{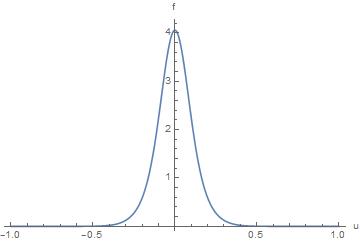}
    \caption{Density (\ref{fnn}) for $N=3$ (left) and $N=4$ (right) at the point $(x,t)=(0.5,0.1)$. Example \ref{ex_bb}.}
		\label{x05t01KL}
    \end{center}
  \end{figure}
	
\begin{figure}[H]
  \begin{center}
    \includegraphics[width=7cm]{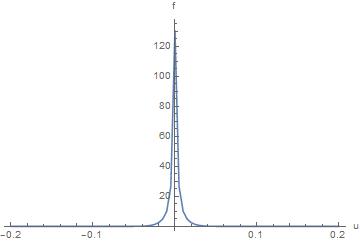}
		\includegraphics[width=7cm]{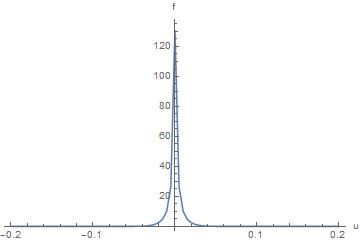}
    \caption{Density (\ref{fnn}) for $N=3$ (left) and $N=4$ (right) at the point $(x,t)=(0.7,0.3)$. Example \ref{ex_bb}.}
		\label{x07t03KL}
    \end{center}
  \end{figure}
	
	\begin{figure}[H]
  \begin{center}
    \includegraphics[width=7cm]{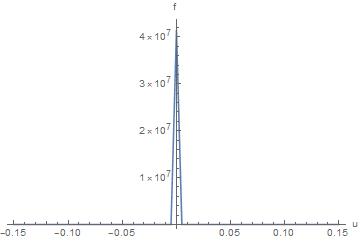}
		\includegraphics[width=7cm]{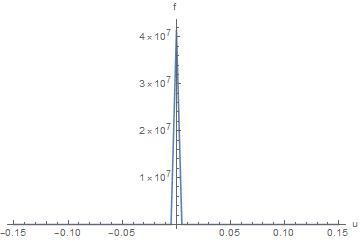}
    \caption{Density (\ref{fnn}) for $N=3$ (left) and $N=4$ (right) at the point $(x,t)=(0.7,1)$. Example \ref{ex_bb}.}
		\label{x07t1KL}
    \end{center}
  \end{figure}
	
	\begin{figure}[H]
  \begin{center}
    \includegraphics[width=7cm]{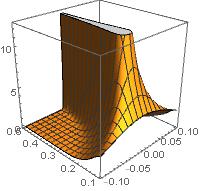}
		\includegraphics[width=7cm]{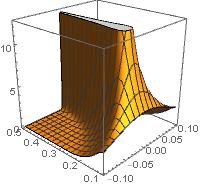}
    \caption{Density (\ref{fbb1}) for $N=3$ at the point $x=0.5$ and $0.1\leq t\leq 0.5$ (left) and density (\ref{fnn}) for $N=4$ at the point $x=0.5$ and $0.1\leq t\leq 0.5$ (right). Example \ref{ex_bb}.}
		\label{3D}
    \end{center}
  \end{figure}

\end{example}
	
\begin{example} \label{ex_nongaussian} \normalfont
We can perform the same analysis for a much larger class of stochastic processes $\phi$. Let $\alpha^2(\omega)$ be, as usual, a random variable such that $\alpha^2(\omega)\geq a>0$ for certain $a\in\mathbb{R}$. Let $\phi$ be a process of the following form:
\begin{equation}
 \phi(x)(\omega)=\sum_{j=1}^\infty \sqrt{\nu_j} \,\sqrt{2}\,\sin(j\pi x)\,\xi_j(\omega), 
\label{phigeneral}
\end{equation}
where the sum is in the topology of $\leb^2([0,1]\times\Omega)$, $\{\nu_j\}_{j=1}^\infty$ are positive real numbers satisfying $\sum_{j=1}^\infty \nu_j<\infty$ and $\{\xi_j\}_{j=1}^\infty$ are absolutely continuous random variables with zero expectation, unit variance and independent. Notice that the sum is well-defined in $\leb^2([0,1]\times\Omega)$, because for two indexes $N>M$ we have, by Pythagoras Theorem in $\leb^2([0,1]\times\Omega)$,
\begin{align*}
 \left\|\sum_{j=M+1}^N \sqrt{\nu_j} \,\sqrt{2}\,\sin(j\pi x)\,\xi_j\right\|_{\leb^2([0,1]\times\Omega)}^2= {} & \sum_{j=M+1}^N \nu_j\, \|\sqrt{2}\,\sin(j\pi x)\|_{\leb^2([0,1])}^2\|\xi_j\|_{\leb^2(\Omega)}^2 \\
= {} & \sum_{j=M+1}^N \nu_j\stackrel{N,M\rightarrow\infty}{\longrightarrow}0. 
\end{align*}

Expression (\ref{phigeneral}) for $\phi$ is very intuitive: as we require $\phi(0)=\phi(1)=0$, the orthonormal basis to work with in order to expand $\phi(\cdot)(\omega)$ as a random Fourier series is $\{\sqrt{2}\,\sin(j\pi x)\}_{j=1}^\infty$. In this way,
\[ \phi(x)(\omega)=\sum_{j=1}^{\infty} c_j(\omega)\,\sqrt{2}\,\sin(j\pi x). \]

Expression (\ref{phigeneral}) corresponds to the Karhunen-Lo\`{e}ve expansion (\ref{klphi}), due to the uniqueness of it\footnote{Let $\{X(t):\,t\in\mathcal{T}\subseteq\mathbb{R}\}$ be a stochastic process in $\leb^2(\mathcal{T}\times\Omega)$. Suppose that $X(t)(\omega)=\sum_{j=1}^{\infty} \sqrt{\nu_j}\,\phi_j(t)\xi_j(\omega)$ in the sense of $\leb^2(\mathcal{T}\times\Omega)$. Suppose that $\{\phi_j\}_{j=1}^\infty$ is an orthonormal basis of $\leb^2(\mathcal{T})$ and $\xi_1,\xi_2,\ldots$ have zero expectation, unit variance and are pairwise uncorrelated. Then the series corresponds to the Karhunen-Lo\`{e}ve expression of $X$. Indeed, we just need to prove that $\mathcal{C}\phi_k=\nu_k\phi_k$, $k\geq1$. We have $\Cov[X(t),X(s)]=\sum_{j=1}^\infty \nu_j\,\phi_j(t)\,\phi_j(s)$. Then $\mathcal{C}\phi_k(t)=\int_{\mathcal{T}}\Cov[X(t),X(s)]\phi_k(s)\,ds=\sum_{j=1}^\infty \nu_j\,\phi_j(t)\int_{\mathcal{T}}\phi_j(s)\phi_k(s)\,ds=\nu_k\,\phi_k(t)$.}.

If $\phi$ has expression (\ref{phigeneral}) and the density function $f_{\xi_1}$ is Lipschitz on $\mathbb{R}$, then the hypotheses of Theorem \ref{teor2} and Theorem \ref{teor3} hold. Indeed, 
\begin{align}
 A_n(\omega)= {} & 2\int_0^1 \phi(y)(\omega)\sin(n\pi y)\,dy=2\sum_{j=1}^\infty \sqrt{\nu_j}\,\sqrt{2}\,\int_0^1 \sin(j\pi y)\sin(n\pi y)\,dy\,\xi_j(\omega) \nonumber \\
= {} & \sqrt{2}\,\sqrt{\nu_n}\,\xi_n(\omega), \label{A1x1}
\end{align}
so $A_1,A_2,\ldots$ are absolutely continuous and independent, which gives the hypothesis of Theorem \ref{teor2}. Also, if $\alpha^2(\omega)\leq b$ for certain $b\in\mathbb{R}$ and we denote $\phi_1(x)=\sqrt{2}\,\sin(\pi x)$ the first eigenfunction in (\ref{phigeneral}), we have
\[\hat{\phi}_1(n)=\sqrt{2}\int_0^1\sin(\pi y)\sin(n\pi y)\,\dif y=\begin{cases} 0,&\; n\neq1,\\ \frac{\sqrt{2}}{2},&\; n=1,\end{cases}\]
\[T_N(\phi_1)(x,t,\alpha^2(\omega))=\frac{\sqrt{2}}{2}\e^{-\pi^2 \alpha^2(\omega)t}\sin(\pi x)\geq\frac{\sqrt{2}}{2}\e^{-\pi^2 b\,t}\sin(\pi x)=:C(x,t)>0,  \]
for $0<x<1$ and $t>0$. This gives the hypothesis of Theorem \ref{teor3}.

Thus, we can use formulas (\ref{fr}), $f_{u_N(x,t)}(u)$, and (\ref{fnn}), $f_{u_{N,N}(x,t)}(u)$, to approximate the density of the solution $u(x,t)(\omega)$ given in (\ref{sol2}), whenever $\phi$ has the form (\ref{phigeneral}).

Let us explore a non-Gaussian process $\phi$. For example,
\[ \phi(x)(\omega)=\sum_{j=1}^{\infty} \frac{\sqrt{2}}{j^{\frac32}\sqrt{1+\log j}}\sin(j\pi x)\xi_j(\omega), \]
where $\nu_j=1/(j^3(1+\log j))$ and $\xi_1,\xi_2,\ldots$ are identically distributed and independent with 
\[ f_{\xi_1}(\xi_1)=\frac{\sqrt{2}}{\pi(1+\xi_1^4)} \]
(it can be checked that $f_{\xi_1}$ is a density function, Lipschitz on $\mathbb{R}$, such that its expectation is $0$ and variance is $1$). Let $\alpha^2\sim\text{Uniform}(1,2)$.

In Figures \ref{x05t01gen}, \ref{x07t03gen} and \ref{x07t1gen}, we show the density $f_{u_N(x,t)}(u)$ given in (\ref{fr}) for $N=2$ (left) and $N=3$ (right) at the points $(x,t)=(0.5,0.1)$, $(x,t)=(0.7,0.3)$ and $(x,t)=(0.7,1)$, respectively. In Figures \ref{x05t01KLgen}, \ref{x07t03KLgen} and \ref{x07t1KLgen}, we see the density $f_{u_{N,N}(x,t)}(u)$ given in (\ref{fnn}) for $N=3$ (left) and $N=4$ (right) at the same points as before. In Figure \ref{3Dgen}, three dimensional plots of the density $f_{u_3(x,t)}(u)$ given in (\ref{fr}) (left) and of the density $f_{u_{4,4}(x,t)}(u)$ given in (\ref{fnn}) (right) are presented, with $x=0.5$ fixed and $t\in [0.1,0.5]$ varying, to show the time evolution of the density.

In Table \ref{taulaL12}, we compare the two plots in each of the figures in order to assess convergence. In Table \ref{taulaEV2}, we approximate the expectation and variance of $u(x,t)(\omega)$ at the previous points.
	
\begin{table}[H]
\begin{center}
\begin{tabular}{|c|c|c|c|} \hline
$\leb^1$ / $(x,t)$ & $(0.5,0.1)$ & $(0.7,0.3)$ & $(0.7,1)$  \\ \hline
$\|f_{u_2(x,t)}-f_{u_3(x,t)}\|_{\leb^1(\mathbb{R})}$ & $3.27465\cdot 10^{-8}$ & $1.44267\cdot 10^{-8}$ & $0.000880742$ \\ \hline
$\|f_{u_{3,3}(x,t)}-f_{u_{4,4}(x,t)}\|_{\leb^1(\mathbb{R})}$ & $2.11166\cdot 10^{-8}$ & $4.51255\cdot 10^{-8}$ & $0.00346924$ \\ \hline
\end{tabular}
\caption{Comparison of the two plots in each of the figures. Example \ref{ex_nongaussian}.}
\label{taulaL12}
\end{center}
\end{table}

\begin{table}[H]
\begin{center}
\begin{tabular}{|c|c|c|c|} \hline
$\mathbb{E}$, $\mathbb{V}$ / $(x,t)$ & $(0.5,0.1)$ & $(0.7,0.3)$ & $(0.7,1)$  \\ \hline
$\mathbb{E}[u_3(x,t)]$ & $1.48536\cdot 10^{-17}$ & $8.5652\cdot 10^{-18}$ & $2.56555\cdot 10^{-24}$ \\ \hline
$\mathbb{V}[u_3(x,t)]$ & $0.100146$ & $0.000486562$ & $2.01166\cdot 10^{-12}$ \\ \hline
$\mathbb{E}[u_{4,4}(x,t)]$ & $-5.31302\cdot 10^{-17}$ & $1.19262\cdot 10^{-18}$ & $6.33678\cdot 10^{-25}$ \\ \hline
$\mathbb{V}[u_{4,4}(x,t)]$ & $0.100146$ & $0.00048651$ & $2.01166\cdot 10^{-12}$ \\ \hline
\end{tabular}
\caption{Simulation of the expectation and variance. Example \ref{ex_nongaussian}.}
\label{taulaEV2}
\end{center}
\end{table} 

\begin{figure}[H]
  \begin{center}
    \includegraphics[width=7cm]{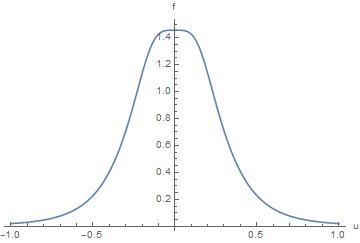}
		\includegraphics[width=7cm]{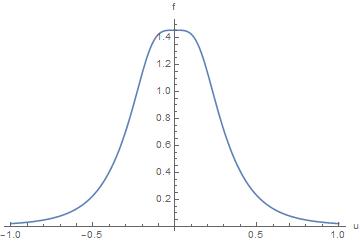}
    \caption{Density (\ref{fr}) for $N=2$ (left) and $N=3$ (right) at the point $(x,t)=(0.5,0.1)$. Example \ref{ex_nongaussian}.}
		\label{x05t01gen}
    \end{center}
  \end{figure}
	
\begin{figure}[H]
  \begin{center}
    \includegraphics[width=7cm]{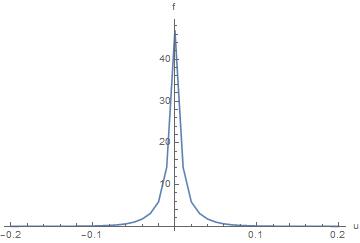}
		\includegraphics[width=7cm]{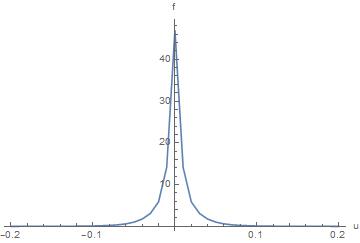}
    \caption{Density (\ref{fr}) for $N=2$ (left) and $N=3$ (right) at the point $(x,t)=(0.7,0.3)$. Example \ref{ex_nongaussian}.}
		\label{x07t03gen}
    \end{center}
  \end{figure}
	
	\begin{figure}[H]
  \begin{center}
    \includegraphics[width=7cm]{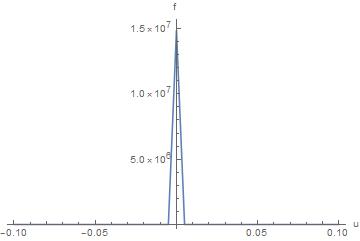}
		\includegraphics[width=7cm]{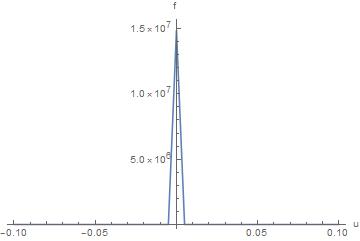}
    \caption{Density (\ref{fr}) for $N=2$ (left) and $N=3$ (right) at the point $(x,t)=(0.7,1)$. Example \ref{ex_nongaussian}.}
		\label{x07t1gen}
    \end{center}
  \end{figure}

	\begin{figure}[H]
  \begin{center}
    \includegraphics[width=7cm]{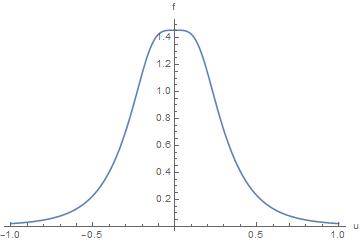}
		\includegraphics[width=7cm]{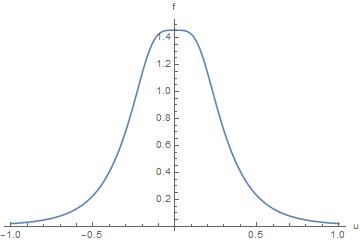}
    \caption{Density (\ref{fnn}) for $N=3$ (left) and $N=4$ (right) at the point $(x,t)=(0.5,0.1)$. Example \ref{ex_nongaussian}.}
		\label{x05t01KLgen}
    \end{center}
  \end{figure}
	
\begin{figure}[H]
  \begin{center}
    \includegraphics[width=7cm]{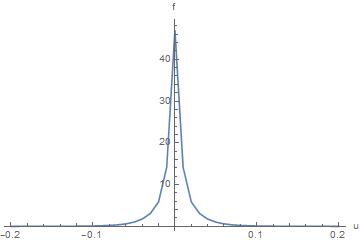}
		\includegraphics[width=7cm]{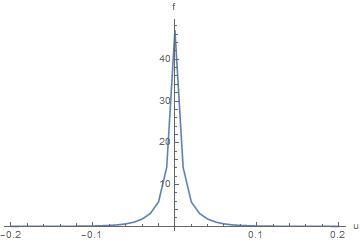}
    \caption{Density (\ref{fnn}) for $N=3$ (left) and $N=4$ (right) at the point $(x,t)=(0.7,0.3)$. Example \ref{ex_nongaussian}.}
		\label{x07t03KLgen}
    \end{center}
  \end{figure}
	
	\begin{figure}[H]
  \begin{center}
    \includegraphics[width=7cm]{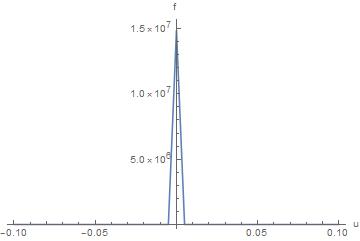}
		\includegraphics[width=7cm]{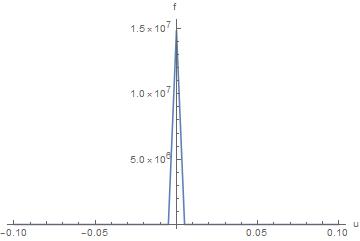}
    \caption{Density (\ref{fnn}) for $N=3$ (left) and $N=4$ (right) at the point $(x,t)=(0.7,1)$. Example \ref{ex_nongaussian}.}
		\label{x07t1KLgen}
    \end{center}
  \end{figure}
	
	\begin{figure}[H]
  \begin{center}
    \includegraphics[width=7cm]{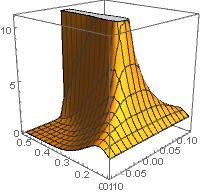}
		\includegraphics[width=7cm]{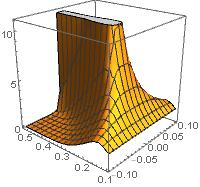}
    \caption{Density (\ref{fr}) for $N=3$ at the point $x=0.5$ and $0.1\leq t\leq 0.5$ (left) and density (\ref{fnn}) for $N=4$ at the point $x=0.5$ and $0.1\leq t\leq 0.5$ (right). Example \ref{ex_nongaussian}.}
		\label{3Dgen}
    \end{center}
  \end{figure}

\end{example}

\begin{example}	\label{ex_counterex} \normalfont
The necessity of the Lipschitz condition for $f_{A_1}$ in Theorem \ref{teor2} and for $f_{\xi_1}$ in Theorem \ref{teor3} can be analyzed numerically. Consider $\alpha^2\sim \text{Uniform}(1,2)$ and
\[ \phi(x)(\omega)=\sum_{j=1}^{\infty} \frac{\sqrt{2}}{j^{\frac32}\sqrt{1+\log j}}\sin(j\pi x)\xi_j(\omega), \]
where $\xi_1,\xi_2,\ldots$ are independent with uniform distribution on $(-\sqrt{3},\sqrt{3})$. We have that $\xi_1,\xi_2,\ldots$ have zero expectation with unit variance, but $f_{\xi_1}$ is not Lipschitz on $\mathbb{R}$, since it has a jump discontinuity at $\pm \sqrt{3}$. By (\ref{A1x1}) and Lemma \ref{lema_abscont}, 
\[ f_{A_1}(a_1)=\frac{1}{\sqrt{2}} f_{\xi_1}\left(\frac{a_1}{\sqrt{2}}\right). \]
This density function is neither Lipschitz. In Figure \ref{cN} and Table \ref{taulaL13}, it seems that density (\ref{fr}), $f_{u_N(x,t)}(u)$, does not converge. Although this is not an analytical proof, the example shows that the absence of the Lipschitz condition changes the convergence results of the numerical experiments.

\begin{table}[H]
\begin{center}
\begin{tabular}{|c|c|} \hline
$\leb^1$ / $(x,t)$ & $(0.5,0.3)$   \\ \hline
$\|f_{u_2(x,t)}-f_{u_3(x,t)}\|_{\leb^1(\mathbb{R})}$ & $0.19156$ \\ \hline
$\|f_{u_{3}(x,t)}-f_{u_{4}(x,t)}\|_{\leb^1(\mathbb{R})}$ & $1.86146$ \\ \hline
\end{tabular}
\caption{Comparison of the three plots in Figure \ref{cN}. Example \ref{ex_counterex}.}
\label{taulaL13}
\end{center}
\end{table}

\begin{figure}[H]
  \begin{center}
    \includegraphics[width=7cm]{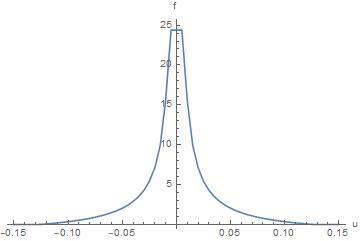}
		\includegraphics[width=7cm]{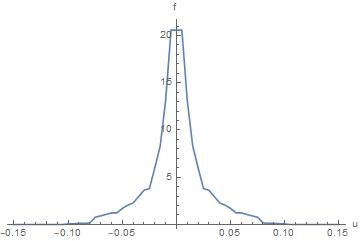}
		\includegraphics[width=7cm]{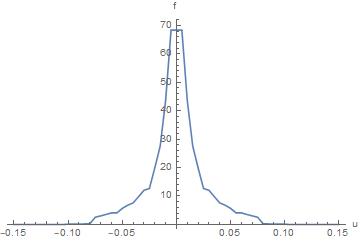}
    \caption{Density (\ref{fr}) for $N=2$ (up left), $N=3$ (up right) and $N=4$ (down) at the point $(x,t)=(0.5,0.3)$. Example \ref{ex_counterex}.}
		\label{cN}
    \end{center}
  \end{figure}

\end{example}

\section{Conclusions}

In this paper we have determined approximations of the probability density function of the solution of the randomized heat equation with homogeneous boundary conditions. This solution is a stochastic process expressed as a random series, which is obtained via the classical method of separation of variables. Three theorems, \ref{teor1}, \ref{teor2} and \ref{teor3}, illustrate the theoretical ideas of the paper. In Theorem \ref{teor1} and Theorem \ref{teor2}, we have focused on the hypotheses on the joint density of the random Fourier coefficients appearing in the random series, as well as other hypotheses on the random diffusion coefficient. In Theorem \ref{teor3}, we focused on the hypotheses on the Karhunen-Lo\`{e}ve expansion of the initial condition process, as well as other assumptions on the random diffusion coefficient. A very important hypothesis in Theorem \ref{teor2} and Theorem \ref{teor3} is concerned with a Lipschitz condition. The hypotheses of the three theorems have been established in order to prove that the approximating density functions form a uniformly Cauchy sequence. As we have seen, Theorem \ref{teor2} and Theorem \ref{teor3} give a great variety of examples. The numerical experiments evince that, under the assumptions of the theorems, the two approaches offer very similar results and a very quick convergence of the approximating density functions. The last example demonstrates numerically the necessity of the Lipschitz condition set in Theorem \ref{teor2} and Theorem \ref{teor3}.

\section*{Acknowledgements}
This work has been  supported by the Spanish Ministerio de Econom\'{i}a y Competitividad grant MTM2013-41765-P.

\section*{Conflict of Interest Statement} 
The authors declare that there is no conflict of interests regarding the publication of this article.

\end{document}